\documentclass[11pt]{amsart}

\usepackage{amsfonts}

\usepackage{bm}
\usepackage{dsfont}
\usepackage{amsmath}
\usepackage{MnSymbol}

\newtheorem{theorem}{Theorem}[section]
\newtheorem{lemma}[theorem]{Lemma}

\newtheorem{assumption}[theorem]{Assumption}

\theoremstyle{definition}

\theoremstyle{remark}

\everymath{\displaystyle}

\title[Asymptotics of hydroelastic waves]{Asymptotics of two-dimensional hydroelastic waves: the zero mass, zero bending limit}

\author{Shunlian Liu}
\address{School of Science, Hunan University of Technology Zhuzhou, Hunan 412007, China}
\email{06yslsl@163.com}
\author{David M. Ambrose}
\address{Department of Mathematics, Drexel University, Philadelphia, PA 19104 USA}
\email{dma68@drexel.edu}

\begin{document}

\begin{abstract}
We consider two-dimensional hydroelastic waves, in which a free fluid surface separates two fluids of infinite vertical extent.
Elastic effects are accounted for at the interface, with a parameter measuring the elastic bending force and another parameter
measuring the mass of the elastic sheet.  In prior work, the authors have demonstrated well-posedness of this initial value
problem in Sobolev spaces.  We now take the limit as these two parameters vanish.  Since the size of the time interval
of existence given by this prior theory vanishes as the mass and bending parameters go to zero,  we now establish
estimates which are uniform with respect to these parameters.  We may then make an additional estimate which demonstrates
that the solutions form a Cauchy sequence as the parameters go to zero, so that the limit may be taken.  This demonstrates that
the vortex sheet with surface tension is the zero mass, zero bending limit of hydroelastic waves in two spatial dimensions.

\end{abstract}

\maketitle

\section{Introduction}

We study a singular limit of an interfacial fluid flow problem, accounting for hydroelastic and capillary effects at the boundary
between two two-dimensional, irrotational, incompressible, inviscid fluids.  As such, these two fluids each satisfy the incompressible
Euler equations in the interior of the fluid region.  The fluids are of infinite vertical extent, and we consider periodic boundary conditions
in the horizontal direction.  The hydroelastic bending term is the highest-order effect at the free surface, and we study the asymptotic
limit in which this coefficient goes to zero.  We simultaneously send the mass of the sheet to zero.  We show that in this limit, keeping
the positive coefficient of surface tension fixed, the solutions of the hydroelastic problem approach the solution of the vortex sheet with
surface tension.

The hydroelastic fluid problem models fluid flow with a free surface in cases in which a relatively thin elastic solid is present, such
as ice sheets on the ocean or the motion of a flapping flag surrounded by air.  This was modeled by Plotnikov and Toland using the
special Cosserat theory of elastic shells \cite{plotnikovToland}.  A number of analytical studies of existence of solitary and periodic
traveling waves were then performed, including cases with and without accounting for the mass of the sheet, and in two and three
spatial dimensions \cite{groves1}, \cite{ZAMP}, \cite{EJAM2}, \cite{baldiToland}, \cite{groves2}, \cite{toland-heavy}, \cite{toland}.
The second author and Siegel proved well-posedness of the initial value problem in the two-dimensional
case without mass \cite{ambroseSiegel-hydroelastic},
and the authors proved well-posedness for the case with mass \cite{liu2017well}.  Wang and Yang proved well-posedness
of the initial value problem in three spatial dimensions \cite{wangYang}.
Many numerical studies of
the traveling hydroelastic wave problem have also been made, such as \cite{gaoEtAl1}, \cite{gaoEtAl2}, \cite{guyenneParau},
\cite{milewskiEtAl}, \cite{milewskiWang}, \cite{parau}, \cite{wang2}, \cite{wang1}.

While the authors are unaware of other studies of the zero bending limit of interfacial waves, there are prior studies on the
 zero surface tension limit of problems in free-surface fluid dynamics.  The second author
and Masmoudi showed in \cite{ambroseMasmoudi1}, \cite{ambroseMasmoudi3} that the zero surface tension limit of
water waves may be taken, in two and three space dimensions, respectively.
Agrawal took the zero surface tension limit of water waves with singular data (i.e. in the presence of angled crests) as well
\cite{agrawal}.
The authors have taken the zero
surface tension limit of interfacial Darcy flow in two and three space dimensions \cite{ambroseDarcy3}, \cite{liuAmbrose3}.
The zero surface tension limit of interfacial Darcy flow was also considered by Flynn and Nguyen, at low regularity, in
\cite{flynnNguyen}.  Coutand, Hole, and Shkoller studied the zero surface tension limit for free-boundary problems for the
compressible Euler equations \cite{shkoller1}, and
Had\v{z}i\'{c} and Shkoller have taken the zero surface tension limit for the Stefan problem \cite{shkoller2}.
It is not always the case, however, that such limits may be taken.  In a series of papers it has been shown that there are situations
for Darcy/Hele-Shaw flow in which the zero surface tension limit is not the same as the solution without surface tension.
In particular, so-called daughter singularities arise \cite{houDaughter}, \cite{houDaughter2}, \cite{tanveerDaughter},
\cite{siegelDaughter}.  Thus it can be important to rigorously establish such limits when possible, to remove any doubt
as to the regularity of the solutions with respect to the parameters.

The method used for analysis of the present problem is rooted in the numerical work of Hou, Lowengrub, and Shelley (HLS) for the
efficient computation of vortex sheets with surface tension \cite{HLS1}, \cite{HLS2}.  The key features of the HLS formulation are
(i) that geometric dependent variables related to the physics are evolved rather than the Cartesian coordinates of the interface,
(ii) a normalized arclength parameterization is maintained at all times by careful choice of an artificial tangential velocity, and
(iii) singular integrals such as the Birkhoff-Rott integral are approximated by simpler operators such as the Hilbert transform (as was
also done in, for instance, \cite{BHLConvergence}, \cite{BHL}).  These choices have also proven to be well-suited for
rigorous analysis, as has been demonstrated in the series of works \cite{ambroseRealThesis}, \cite{ambroseThesis}, \cite{ambroseDarcy1}, \cite{ambroseDarcy3},
\cite{ambroseMasmoudi1}, \cite{ambroseSiegel-hydroelastic},
\cite{liu2017well}, as well as works by other authors such as \cite{hurEtAl}, \cite{cordobaWave}, \cite{cordobaMuskat}, \cite{duell}, \cite{guoEtAl},  \cite{tofts}
among others.

The plan of the paper is as follows.    In Section \ref{equationsSection}, we develop the equations of motion for the interfacial hydroelastic
wave problem in the HLS framework.  In Section \ref{preliminarySection} we give the estimates on approximation of the Birkhoff-Rott
integral by the Hilbert transform, and related results for other operators.  In Section \ref{uniformSection} we demonstrate energy estimates
in Sobolev spaces for the interfacial hydroelastic problem which are uniform with respect to the bending coefficient and mass density of
the sheet; this implies existence of solutions on an interval of time which is uniform with respect to these parameters.  In Section
\ref{limitSection}, we then take the limit as these parameters vanish, recovering solutions of the vortex sheet with surface tension.

\section{Equations of motion }\label{equationsSection}
The model of the elastic sheet we consider accounts for mass of the interface; we let $\rho_{0}$ be the corresponding
mass density of the sheet.  The upper and lower fluid of course each have their own densities, which we denote as $\rho_{1}$
and $\rho_{2}$ for the lower and upper fluids, respectively.
We denote the position of the free surface (which we note is one-dimensional) by
$\bm{X}(\alpha,t)=(x(\alpha,t),y(\alpha,t)).$  Here we have taken
 $\alpha$ to be the parameter along the curve and as usual we have let $t$ be the temporal variable.
 We take this surface to be horizontally periodic with period $2\pi,$
meaning that the surface satisfies
\begin{align}\nonumber
x(\alpha+2\pi,t)=x(\alpha,t)+2\pi, \quad y(\alpha+2\pi,t)=y(\alpha,t),
\end{align}
for all $\alpha$ and $t.$
We take a frame of unit tangent and normal vectors at each point along the surface, denoted $\bm{t},\bm{n},$ and we
let $s$ represent arclength of the curve as measured from a given point.
We have the following definitions for the vectors $\bm{t}$ and $\bm{n},$ and for the arclength element $s_\alpha:$
\begin{equation}\nonumber
\bm{t}=\frac{(x_\alpha,y_\alpha)}{s_\alpha},\qquad \bm{n}=\frac{(-y_\alpha,x_\alpha)}{s_\alpha},
\end{equation}
\begin{equation}\nonumber
s_\alpha^2=x_\alpha^2+y_\alpha^2.
\end{equation}
We take the surface to evolve according to a normal velocity, $U,$ and a tangential velocity, $V,$ i.e.
\begin{align}\label{Xt}
\bm{X}_t=U \bm{n}+V\bm{t}.
\end{align}

Rather than evolve the Cartesian coordinates $(x,y)$ of the interface, inspired by the numerical works \cite{HLS1}, \cite{HLS2},
and subsequent analytical works such as \cite{ambroseRealThesis}, we will work with geometric dependent variables.
We have already defined the arclength element, $s_{\alpha},$ and now we introduce the tangent angle the free surface forms with
the horizontal,  $\theta=\tan^{-1}(y_\alpha/x_\alpha).$
The curvature of the surface $\bm{X},$ which we denote by $\kappa,$ has a simple expression in terms of $\theta$ and
$s_{\alpha},$
\begin{align}\nonumber
\kappa=\theta_\alpha/s_\alpha.
\end{align}
In terms of $\theta,$ the tangent and normal vectors may be written as
\begin{align}\label{tnDef}
\bm{t}=(\cos\theta,\sin\theta),\qquad \bm{n}=(-\sin\theta,\cos\theta).
\end{align}

While the normal velocity, $U,$ will be determined by the underlying physics, we are free to choose the tangential velocity to
enforce a convenient parameterization.  We choose a normalized arclength parameterization, so that
$s_{\alpha}(\alpha,t)=L(t)/2\pi$ for all $\alpha$ and $t,$ i.e. so that $s_{\alpha}$ does not depend on $\alpha.$  Here, $L(t)$
is the length of one period of the interface at time $t.$
We may use the expression for horizontal periodicity, $x(2\pi,t)-x(0,t)=2\pi,$ and the expressions
\begin{align}\label{xya}
x_\alpha=\frac{L}{2\pi}\cos(\theta),\qquad y_\alpha=\frac{L}{2\pi}\sin(\theta),
\end{align}
to find
\begin{align}\label{LoperatorDefinition}
L=\frac{4\pi^2}{\int_0^{2\pi}\cos(\theta)\ d\alpha}\ge 2\pi.
\end{align}
We may regard \eqref{LoperatorDefinition} as defining $L$ as an operator acting on $\theta,$ i.e. $L=L[\theta].$
Note that $\theta$ must also satisfy
\begin{equation}\label{zeroMeanSin}
\int_{0}^{2\pi}\sin(\theta)\ d\alpha=0,
\end{equation}
as a further consequence of our periodicity assumption.
We denote the mean of a periodic function, $f,$ by
\begin{equation}\nonumber
\llangle f \rrangle = \frac{1}{2\pi}\int_{0}^{2\pi}f(\alpha)\ d\alpha,
\end{equation}
so that our requirement \eqref{zeroMeanSin} may be restated as $\llangle\sin(\theta)\rrangle=0.$

The evolution \eqref{Xt} implies evolution equations for $s_{\alpha}$ and $\theta,$ which are
\begin{align}
s_{\alpha t}= V_\alpha -\theta_\alpha U,\label{sat}\\
\nonumber
\theta_t=\frac{U_\alpha+V\theta_\alpha}{s_\alpha}.
\end{align}

We now determine the tangential velocity which enforces our normalized arclength parameterization;  since our parameterization is
\begin{align}\nonumber
s_{\alpha}=L/2\pi,
\end{align}
then we will need $s_{\alpha t}$ to satisfy
\begin{align}\label{sa}
s_{\alpha t}=L_t/2\pi.
\end{align}
Using \eqref{sa} together with \eqref{sat}, we may solve for $V_{\alpha};$ we find
\begin{equation}\nonumber
V_{\alpha}=\frac{L_{t}}{2\pi}+\theta_{\alpha}U.
\end{equation}
We then see that $L_{t}=2\pi s_{\alpha t},$ and averaging this over our interval $[0,2\pi],$ we find
\begin{align}\label{LtDef}
L_t=\int_0^{2\pi} s_{\alpha t} d\alpha=-\int_0^{2\pi} \theta_\alpha U\ d\alpha.
\end{align}
We then conclude
\begin{align}\nonumber
V_\alpha=-\frac1{2\pi}\int_0^{2\pi} \theta_\alpha Ud\alpha +\theta_\alpha U
=\mathbb{P}(\theta_\alpha U),
\end{align}
where $\mathbb{P}$ is the projection which removes the mean of a periodic function.
We may then give the equation for $V,$
\begin{align}\label{VDef}
V=\partial_\alpha^{-1}\mathbb{P}(\theta_\alpha U)+V_0(t),
\end{align}
where $\partial_{\alpha}^{-1}$ is the mean zero antiderivative of a periodic function with zero mean.
We have introduced $V_{0}(t),$ which is the mean of $V$ at any time, and we will comment on this by the end of Section \ref{br}.

We now address the physical principles which will allow us to deduce the normal velocity, $U.$
We are considering irrotational fluids, but there is singular vorticity present at the free surface.  That is, the two fluids are
each irrotational in the bulk, but the velocity is discontinuous across the free surface, and upon taking the curl of the velocity
one finds the vorticity is proportional to the Dirac mass of the curve separating the two fluids.  Said another way, we take
a vortex sheet formulation of the problem.  The Birkhoff-Rott integral,
$\bm{W}=(W_1,W_2),$ determines the normal velocity of the interface,
\begin{align}\label{UDef}
U=\bm{W}\cdot \bm{n}.
\end{align}
We will give the formula for the Birkhoff-Rott integral,
and investigate approximations for it,  in Section \ref{br} below.
In a Lagrangian formulation of the problem (rather than our formulation in which we use an artificial tangential velocity), the
tangential velocity of the interface (at least in the density-matched case) would be $\bm{W}\cdot\bm{t};$ the difference between the two tangential velocities is
an important quantity, so we denote
\begin{align}\label{VwDef}
V_W=V-\bm{W}\cdot\bm{t}.
\end{align}
Having introduced this notation, we may rewrite $\theta_{t}$ as
\begin{equation}\label{theta_t}
\theta_t=\frac{1}{s_\alpha}(\bm{W}_{\alpha}\cdot \bm{n}-V_W\theta_\alpha).
\end{equation}
Here, we have used the geometric identity
$\bm{n}_{\alpha}=-\theta_{\alpha}\bm{t}.$

As we have said, there is a jump in tangential velocity across the interface between the two fluids, and we denote this jump
by $\gamma.$  Then $\gamma$ is also the vortex sheet strength, i.e. when we have said that the vorticity is proportional to the
Dirac mass of the curve, this amplitude is also given by $\gamma.$
An evolution equation for $\gamma$ may be inferred from the irrotational, incompressible Euler equations. Specifically
the potential in each fluid satisfies a Bernoulli equation, and we may take the difference of the
limit of these approaching the interface to find
the equation for the jump in velocity potential.  Then, differentiating the result, one finds the equation for $\gamma,$
\begin{equation}\label{gamma}
\gamma_t=-\frac2{\rho_1+\rho_2}[P]_\alpha +\frac{(V_W\gamma)_\alpha}{s_\alpha}
-2A\left(s_\alpha \bm{W}_t\cdot \bm{t}+\frac18\partial_\alpha\left(\frac{\gamma^2}{s_\alpha^2}\right)
-V_W\bm{W}_\alpha\cdot \bm{t}+g y_\alpha\right),
\end{equation}
where $[P]=p_1-p_2$ is the jump in pressure across the interface, $g$ is the constant
acceleration due to gravity, and the Atwood number is $A=\frac{\rho_1-\rho_2}{\rho_1+\rho_2}$.
The equation \eqref{gamma} can be found in several places, such as \cite{ambroseMasmoudi1},
or \cite{ambroseSiegel-hydroelastic}, or \cite{BMO}.
A formula for the jump in pressure is needed to close the model, and we use the expression developed in the model of
Plotnikov and Toland \cite{plotnikovToland}.
The equation for $[P]$ is
\begin{equation}\label{pressure2}
[P]=-\tau \kappa+\rho \partial_{tt}\bm{X}\cdot\bm{n}+\frac12 \sigma\left\{\kappa_{ss} + \frac12\kappa^3\right\}+g\rho \bm{j}\cdot\bm{n}.
\end{equation}
where $\rho=\rho_0/s_\alpha$ is the actual density of the deformed sheet, $\sigma$ is the constant bending modulus,
and $\tau>0$ is the surface tension coefficient.
Using \eqref{pressure2} to substitute for $[P]_\alpha,$ we have
\begin{align*}
\gamma_t&=-\frac{1}{\rho_1+\rho_2}  \sigma \left(\kappa_{ss} + \frac12\kappa^3\right)_\alpha +\frac{2\tau\kappa_\alpha}{\rho_1+\rho_2}+\frac{(V_W\gamma)_\alpha}{s_\alpha}
\\
&-2A\left(s_\alpha \bm{W}_t\cdot \bm{t}+\frac18\partial_\alpha\left(\frac{\gamma^2}{s_\alpha^2}\right)-V_W\bm{W}_\alpha\cdot \bm{t}+g y_\alpha\right)\\
&-\frac{2\rho}{\rho_1+\rho_2} \left(\bm{W}_{\alpha  t}\cdot\bm{n}-\bm{W}_t\cdot\bm{t}\theta_\alpha+( V_W)_\alpha\theta_t
+ V_W \theta_{t\alpha}+g( x_\alpha/s_\alpha)_\alpha\right).
\end{align*}
We now rewrite this in terms of the tangent angle, $\theta,$ and length of one period, $L,$ since we have the relationships
$\kappa=\theta_\alpha/s_\alpha,$ $\partial_s=\partial_\alpha/s_\alpha,$ and $s_\alpha=\frac{L}{2\pi}.$
These considerations yield the following equation for $\gamma_{t}:$
\begin{multline}\label{gammaNew}
\gamma_t=-\sigma\bar{A} \left(\partial_\alpha^4\theta +\frac{3\theta_\alpha^2\theta_{\alpha\alpha}}{2} \right)+\lambda\theta_{\alpha\alpha}+\frac{2\pi(V_W\gamma)_\alpha}{L}
\\
-2A\left(s_\alpha \bm{W}_t\cdot \bm{t}+\left(\frac{\pi^2}{L^2}\right)\gamma\gamma_\alpha-V_W\bm{W}_\alpha\cdot \bm{t}+g y_\alpha\right)\\
-2\widetilde{A}\rho\left(\bm{W}_{\alpha  t}\cdot\bm{n}-\bm{W}_t\cdot\bm{t}\theta_\alpha+ (V_W)_\alpha\theta_t
+ V_W \theta_{t\alpha}-\frac{ 2\pi g x_{\alpha\alpha}}{L}\right).
\end{multline}
Here, we have introduced a few groupings of constants, namely
$\widetilde{A}=\frac{1}{\rho_1+\rho_2},$ $\bar{A}=\frac{8\pi^3 }{L^3(\rho_1+\rho_2)},$ and
$\lambda=\frac{4\tau\pi}{L(\rho_1+\rho_2)}.$  Notice that $A$ and $\widetilde{A}$ are independent of both $t$ and $\alpha,$
while $\bar{A}$ and $\lambda$ are time-dependent through their dependence on $L.$

If $\rho=0$  and $\sigma=0$, then the  evolution of $\gamma$ would be
\begin{equation}\nonumber
\begin{array}{l}
\gamma_t=\lambda\theta_{\alpha\alpha}+\frac{2\pi(V_W\gamma)_\alpha}{L}-2A\left(s_\alpha \bm{W}_t\cdot \bm{t}
+\left(\frac{\pi^2}{L^2}\right)\gamma\gamma_\alpha-V_W\bm{W}_\alpha\cdot \bm{t}+g y_\alpha\right).
\end{array}
\end{equation}
This is the equation for $\gamma$ for the vortex sheet with surface tension \cite{ambroseThesis}, \cite{BMO}.

We will be considering the behavior of solutions as $(\rho_0,\sigma)\rightarrow (0,0)$. We will demonstrate that as $(\rho_0,\sigma)$ vanishes, the sequence of solutions corresponding to $(\rho_0,\sigma)$ forms a Cauchy sequence in an appropriately chosen
function space.  These solutions will then be shown to converge to the solution of the vortex sheet with surface tension initial value
problem.

\subsection{The Birkhoff-Rott integral, related operators, and consequences}\label{br}
We have said that the normal velocity of the interface is given by the normal component of the Birkhoff-Rott integral,
$\bm{W}.$  In this section we will give formulas for the Birkhoff-Rott integral, decomposing it to isolate its most singular part.
This will require the introduction of a few operators related to Hilbert transforms.

The introduction of complex notation simplifies the presentation of the Birkhoff-Rott integral, so we
define the complexification map $\mathcal{C}: \mathds{R}^2\rightarrow \mathds{C}$ to be
\begin{align}\nonumber
\mathcal{C}(x,y)=x+iy.
\end{align}
We will frequently denote $z=\mathcal{C}(x,y)=x+iy.$  Given two real two-vectors $(a_{1},b_{1})$ and $(a_{2},b_{2}),$
the dot product may be expressed through the complexification map as
\begin{equation}\nonumber
(a_{1},b_{1})\cdot(a_{2},b_{2})=\mathrm{Re}\left\{\mathcal{C}(a_{1},b_{1})^{*}\mathcal{C}(a_{2},b_{2})\right\},
\end{equation}
where the star denotes complex conjugation.

Using the complexification, the Birkhoff-Rott integral $\bm{W}$ is given by
 \begin{align}\nonumber
\mathcal{C}(\bm{W})^*=W_1-iW_2=\frac1{4\pi i} PV \int_0^{2\pi} \gamma(\alpha^\prime)\cot\left(\frac{1}{2}(z(\alpha)-z(\alpha^\prime))\right)\ d\alpha^\prime.
\end{align}
We also have the following expressions using the complexification:
\begin{align}\nonumber
&\mathcal{C}(\bm{t})=z_\alpha/s_\alpha,\qquad
\mathcal{C}(\bm{n})=iz_\alpha/s_\alpha,\qquad
\bm{W}\cdot \bm{n}=\mathrm{Re}\{\mathcal{C}(\bm{W})^*\mathcal{C}(\bm{n})\}.
\end{align}

Properties of the periodic Hilbert transform \cite{helson} will be important in what follows, so we now introduce it.
For any periodic function $f\in L^{2},$ we may define its periodic Hilbert transform, $Hf,$  through the singular integral
\begin{align}\nonumber
Hf(\alpha)=\frac1{2\pi} PV\int_0^{2\pi} f(\alpha^\prime)\mathrm{cot}\left(\frac12(\alpha-\alpha^\prime)\right)\ d\alpha^\prime.
\end{align}
The Hilbert transform is a multiplier on Fourier space, with symbol $-i\mathrm{sgn}(k),$ where $k$ is the Fourier variable.

Notice that both the Hilbert transform and the Birkhoff-Rott integral are singular integrals, with the same order of singularity.
We will approximate the Birkhoff-Rott integral with a Hilbert transform, and we introduce an operator to be the difference of these.
Later, we will see that the error (given by this operator) is smooth.  We define
\begin{multline}\nonumber
K[z_d]f(\alpha)
\\=\frac1{4\pi i}\int_0^{2\pi} f(\alpha^\prime)
\left[\mathrm{cot}\left(\frac{1}{2}(z_d(\alpha)-z_d(\alpha^\prime))\right)
-\frac{1}{z_\alpha(\alpha^\prime)}\cot\left(\frac12(\alpha-\alpha^\prime)\right)\right]\ d\alpha^\prime.
\end{multline}
We have introduced here a slight modification of the curve $z,$ as this will be convenient,
\begin{align}\nonumber
z_d(\alpha,t)=z(\alpha,t)-z(0,t);
\end{align}
this is convenient because when integrating functions of $\theta$ to find the curve $z,$ there is a question of the constant of integration.  As the constant of integration has no bearing on the value of $\bm{W},$ though, we use $z_{d}$ which eliminates
the constant.
Specifically, $z_d(\alpha,t)-z_d(\alpha^\prime,t)=z(\alpha,t)-z(\alpha^\prime,t)$ and $\partial_\alpha z_d=\partial_\alpha z.$
Furthermore, we may view $z_{d}$ as being defined through
\begin{align}\label{zdd}
z_d(\alpha,t)=\frac{L}{2\pi}\int_0^\alpha\cos(\theta)+i\sin(\theta)d\alpha.
\end{align}

We will frequently need the commutator of the Hilbert transform and
multiplication by a smooth function $\phi$,
\[
[H,\phi]f(\alpha)=H(\phi f)(\alpha)-\phi(\alpha)H(f)(\alpha).
\]
This commutator $[H,\phi]$ is a smoothing operator, and we will give results to this effect in Section
\ref{smoothingLemmasSection} below.

We will be rewriting the $\theta_{t}$ equation \eqref{theta_t} to bring out the leading behaviour, and so will
will need a helpful expression for $\bf{W}_{\alpha}.$  This will make use of the operators we have introduced just above.
We repeat the argument of \cite{ambroseThesis} on this point, which is to differentiate $\bm{W}$ with respect to $\alpha,$
to recognize a convolution-like structure of the kernel in the integral, and then to integrate by parts.  Then decomposing
in terms of our operators, we have
\begin{multline}\nonumber
\mathcal{C}(\bm{W})^*_\alpha
=\frac1{4\pi i} PV \int_0^{2\pi} z_\alpha(\alpha)\partial_{\alpha^\prime}
\left(\frac{\gamma(\alpha^\prime)}{z_\alpha(\alpha^\prime)}\right)\mathrm{cot}
\left(\frac{1}{2}(z(\alpha)-z(\alpha^\prime))\right)\ d\alpha^\prime\\
=z_\alpha K[z_d]\left(\left(\frac{\gamma}{z_\alpha}\right)_\alpha\right)
+\frac{z_\alpha}{2i}\left[H,\frac{1}{z_\alpha^2}\right]
\left(z_\alpha\left(\frac{\gamma}{z_\alpha}\right)_\alpha\right)
+\frac{1}{2 i z_\alpha}H\left(\gamma_\alpha-\gamma\frac{z_{\alpha\alpha}}{z_\alpha} \right).
\end{multline}
Introducing the notation $\bm{m}$ for the following collection of terms,
\begin{align}\nonumber
\mathcal{C}(\bm{m})^*=z_\alpha K[z_d]\left(\left(\frac{\gamma}{z_\alpha}\right)_\alpha\right)
+\frac{z_\alpha}{2i}\left[H,\frac{1}{z_\alpha^2}\right]
\left(z_\alpha\left(\frac{\gamma}{z_\alpha}\right)_\alpha\right),
\end{align}
our decomposition of $\bm{W}_{\alpha}$ is
\begin{align}\label{w_alpha}
\bm{W}_\alpha=\frac{\pi}{L}H(\gamma_\alpha) \bm{n}-\frac{\pi}{L}H(\gamma \theta_\alpha)\bm{t}+\bm{m}.
\end{align}
Note that we have used here that $z_{\alpha}=s_{\alpha}e^{i\theta},$ which implies $z_{\alpha\alpha}=i\theta_{\alpha}z_{\alpha}.$
Using \eqref{w_alpha} in \eqref{theta_t}, the $\theta_{t}$ equation becomes
\begin{equation}\label{evolution of theta}
\theta_t=\frac{2\pi^2}{L^2}H(\gamma_\alpha)+\frac{2\pi}{L}V_W\theta_\alpha+\frac{2\pi}{L}\bm{m}\cdot \bm{n}.
\end{equation}

We may further improve this expression for $\theta_{t}$ by again using \eqref{w_alpha} to expand $V_{W}.$
To begin, we recall the definition of $V_{W}$ in \eqref{VwDef}, and we differentiate, finding
 \[
 \partial_\alpha V_W= V_\alpha-\bm{W}_\alpha\cdot\bm{ t}-(\bm{W}\cdot\bm{n})\theta_\alpha.
 \]
 Using the equations $V_\alpha=\frac{L_t}{2\pi}+\theta_\alpha U$ and $U=\bm{W}\cdot\bm{n}$, this becomes
 \[
 \partial_\alpha V_W=\frac{\pi}{L}H(\gamma\theta_\alpha)-\bm{m}\cdot\bm{t}+\frac{L_t}{2\pi}.
 \]
We may simplify this by applying the projection $\mathbb{P}$ which removes the mean of a periodic function; since
the image of each of $\partial_{\alpha}$ and $H$ has zero mean, and since $\mathbb{P}$ applied to a constant function is zero,
we have
\[
\partial_\alpha V_W =\frac{\pi}{L}H(\gamma\theta_\alpha)-\mathds{P}(\bm{m}\cdot\bm{t}).
\]
We next integrate, applying the zero-mean antiderivative operator $\partial_{\alpha}^{-1},$
\begin{equation}\label{V_Wep}
V_W =\partial_\alpha^{-1}\left(\frac{\pi}{L}H(\gamma\theta_\alpha)-\mathds{P}(\bm{m}\cdot\bm{t})\right).
\end{equation}
We mentioned above that we would set the mean of $V$ to be $V_{0}(t),$ and this has now been set.  We have chosen
here that $V_{W}$ should have zero mean, and this means that we have chosen that at each time the mean of $V$ is to be
the same as the mean of $\bm{W}\cdot\bm{t}.$
That is to say, implicit in \eqref{V_Wep} is the choice
\begin{equation}\label{V0Def}
V_0(t)=\frac{1}{2\pi}\int_0^{2\pi}\bm{W}\cdot\bm{t}\ d\alpha.
\end{equation}

\subsection{Calculation of $(\bm{W}_{ t}\cdot \bm{t})s_\alpha$ and  $\bm{W}_{\alpha t}\cdot \bm{n}$}
The quantities $(\bm{W}_{t}\cdot\bm{t})$ and $\bm{W}_{\alpha t}\cdot\bm{n}$ appear on the right-hand side of the
$\gamma_{t}$ equation, so we must find useful expressions for them.
To this end, we write $\bm{W}_t$ as
\begin{equation}\nonumber
\begin{array}{l}
\mathcal{C}(\bm{W}_t)^*=\frac1{4\pi i} PV \int_0^{2\pi} \gamma_t(\alpha^\prime)
\mathrm{cot}\left(\frac{1}{2}(z(\alpha)-z(\alpha^\prime))\right)
\ d\alpha^\prime
\\
-\frac1{8\pi i} PV \int_0^{2\pi} \gamma(\alpha^\prime)(z_t(\alpha)-z_t(\alpha^\prime))
\csc^2\left(\frac{1}{2}(z(\alpha)-z(\alpha^\prime))\right)\ d\alpha^\prime.
\end{array}
\end{equation}
This allows us to form the combination $(\bm{W}_{ t}\cdot \bm{t})s_\alpha,$ which is
\begin{multline}\nonumber
(\bm{W}_{ t}\cdot \bm{t})s_\alpha=\mathrm{Re}\{\mathcal{C}(\bm{W}_t)^*z_\alpha\}=
\\
\mathrm{Re}\left\{ \frac{z_\alpha}{4\pi i} PV \int_0^{2\pi} \gamma_t(\alpha^\prime)\cot\left(\frac{1}{2}(z(\alpha)-z(\alpha^\prime))\right)
\ d\alpha^\prime \right\}
\\
-\mathrm{Re}\left\{\frac{z_\alpha}{8\pi i} PV \int_0^{2\pi} \gamma(\alpha^\prime)(z_t(\alpha)
-z_t(\alpha^\prime))\csc^2\left(\frac{1}{2}(z(\alpha)-z(\alpha^\prime))\right)\ d\alpha^\prime\right\}.
\end{multline}
We introduce an integral operator $J[z_d]\gamma_t$ to be the first term on the right-hand side, which
has the formula
\begin{align}\label{JzDef}
J[z_d]\gamma_t=\mathrm{Re}\left\{z_\alpha K[z_d]\gamma_t+\frac{z_\alpha}{2i}
\left[H,\frac{1}{z_\alpha}\right]\gamma_t\right\}.
\end{align}
We denote by $R_{0}$ the
second term on the right-hand side.  We then integrate $R_{0}$ by parts, resulting in two terms, which we call $R_{1}$ and $R_{2}.$
The formulas for $R_{1}$ and $R_{2}$ are
\begin{multline}\nonumber
R_1=\mathrm{Re}\left\{z_\alpha z_t K[z_d]\left(\left(\frac{\gamma}{z_\alpha}\right)_{\alpha}\right)\right\}
-\mathrm{Re}\left\{z_\alpha K[z_d]\left(z_t\left(\frac{\gamma}{z_\alpha}\right)_\alpha\right)\right\}
\\
-\mathrm{Re}\left\{\frac{z_\alpha}{2i}[H,z_t]\left(\frac{1}{z_\alpha}\left(\frac{\gamma}{z_\alpha}\right)_\alpha\right)\right\},
\end{multline}
\begin{equation}\nonumber
R_{2}=-\mathrm{Re}\left\{z_\alpha K[z_d]\left(\frac{\gamma z_{t\alpha}}{z_\alpha}\right)\right\}
-\mathrm{Re}\left\{\frac{z_\alpha}{2i}\left[H,\frac{1}{z_\alpha}\right]\left(\frac{\gamma z_{t\alpha}}{z_\alpha}\right)\right\}
-\frac{1}{2}H\left(\gamma \theta_t\right).
\end{equation}
We conclude with the expression
\begin{align}\nonumber
(\bm{W}_{ t}\cdot \bm{t})s_\alpha=J[z_d]\gamma_t+R_1+R_2.
\end{align}

We also need a decomposition for $\bm{W}_{\alpha t}\cdot \bm{n}.$  We leave out the intermediate steps, but we arrive
similarly at
\begin{align}\nonumber
\bm{W}_{\alpha t}\cdot\bm{n}=\frac{\pi}{L}H(\gamma_{\alpha t})+S[z_d](\gamma_t)+R_3+R_4,
\end{align}
where
\begin{equation}\label{Sdef}
S[z_d](\gamma_t)=\mathrm{Re}\left\{\frac{i z_\alpha^2}{s_\alpha}K[z_d]\left(\left(\frac{\gamma_{ t}}{z_\alpha}
\right)_\alpha\right)\right\}
+\mathrm{Re}\left\{\frac{z_\alpha^2}{2s_\alpha}
\left[H,\frac{1}{z_\alpha^2}\right]\left(z_\alpha\left(\frac{\gamma_{ t}}{z_\alpha}\right)_\alpha\right)\right\},
\end{equation}
\begin{multline}\nonumber
R_3=\mathrm{Re}\left\{\frac{i z_\alpha^2}{s_\alpha}K[z_d]\left(\left(\frac{-\gamma z_{\alpha t}}{z^2_\alpha}\right)_{\alpha }\right)\right\}
+\mathrm{Re}\left\{\frac{z_\alpha^2}{2s_\alpha}\left[H,\frac{1}{z_\alpha^2}\right]\left(z_\alpha\left(\frac{-\gamma z_{\alpha t}}{z^2_\alpha}\right)_{\alpha }\right)\right\}
\\
+\mathrm{Re}\left\{\frac{-z_\alpha^2}{2s_\alpha}\left[H,\frac{1}{z_\alpha^2}\right]\left(\left(\frac{\gamma}{z_\alpha}\right)_\alpha z_{\alpha t}\right)
-\frac{i z_\alpha^2}{s_\alpha}K[z_d]\left(\left(\frac{\gamma}{z_\alpha}\right)_\alpha \frac{z_{\alpha t}}{z_\alpha}\right)\right\}
\\
+\mathrm{Re}\left\{\frac{iz_\alpha^2 z_t }{s_\alpha} K[z_d]\left(\frac{\left(\gamma/z_\alpha\right)_\alpha}{z_\alpha}\right)_\alpha\right\}
-\mathrm{Re}\left\{\frac{iz_\alpha^2}{s_\alpha} K[z_d]\left(z_t\left(\frac{\left(\gamma/z_\alpha\right)_\alpha}{z_\alpha}\right)_\alpha\right)\right\}
\\
-\mathrm{Re}\left\{\frac{z_\alpha^2}{2s_\alpha}[H,z_t]
\left(\frac{1}{z_\alpha}\left(\frac{\left(\gamma/z_\alpha\right)_\alpha}{z_\alpha}\right)_\alpha\right)\right\},
\end{multline}
\begin{align}\nonumber
R_4=\frac{\pi \theta_t}{L}H(\gamma \theta_\alpha)
-\frac{\pi L_t}{L^2}H(\gamma_\alpha)-\frac{2\pi}{L}H(\gamma \theta_\alpha \theta_t)+ \frac{L_{ t}}{L}\bm{m}\cdot \bm{n}-\theta_t \bm{m}\cdot\bm{t}.
\end{align}

\subsection{The small-scale decomposition}
We will now rewrite the $\theta_{t}$ and $\gamma_{t}$ equations to emphasize the leading-order behavior, as well as to
isolate important contributions from our bending, mass, and surface tension forces (since bending and mass will be going to
zero, leaving surface tension remaining).  Using the expressions we have developed so far, the $\gamma_{t}$ equation
\eqref{gammaNew} becomes
\begin{multline}\nonumber
\gamma_t=-\sigma\bar{A} \left(\partial_\alpha^4\theta +\frac{3\theta_\alpha^2\theta_{\alpha\alpha}}{2} \right)+\lambda\theta_{\alpha\alpha}+\frac{2\pi((V_W)_\alpha\gamma)}{L}+\frac{2\pi(V_W\gamma_\alpha)}{L}
\\
-2\rho\widetilde{A}\left(\frac{\pi}{L}H(\gamma_{\alpha t})+S[z_d](\gamma_t)\right)
-\left(2A-4\pi\rho\widetilde{A}\theta_\alpha/L\right) ( J[z_d](\gamma_t))
\\
-2\rho\widetilde{A}\Bigg(-\frac{4\pi^3}{L^2}[H,(H\gamma_\alpha)]\gamma\theta_\alpha
-(\mathds{P}(\bm{m}\cdot\bm{t})+\bm{m}\cdot\bm{t})\frac{2\pi^2}{L^2}H(\gamma_\alpha)-\frac{\pi L_t}{L^2}H(\gamma_\alpha)
\\
+\frac{2\pi^2}{L^2} V_W H(\gamma_{\alpha\alpha})+\frac{2\pi}{L} V_W ^2\theta_{\alpha\alpha}
\Bigg)
\\
+\left(2A-4p\rho\widetilde{A}\theta_\alpha/L\right)
\left(\frac{\pi^2}{L^2}[H,\gamma] H(\gamma_\alpha)-\frac{\pi^2}{L^2}\gamma\gamma_\alpha+H\left(\frac{\pi}{L} V_W \theta_\alpha\gamma
+\frac{\pi}{L}\gamma \bm{m}\cdot \bm{n}\right)\right)
\\
-\left(2A-4\pi\rho\widetilde{A}\theta_\alpha/L\right)
\left(-\mathrm{Re}\left\{z_\alpha K[z_d]\left(\frac{\gamma z_{t\alpha}}{z_\alpha}\right)\right\}
-\mathrm{Re}\left\{\frac{z_\alpha}{2i}\left[H,\frac{1}{z_\alpha}\right]\left(\frac{\gamma z_{t\alpha}}{z_\alpha}\right)\right\}+R_1\right)
\\
-2\rho\widetilde{A}\left(\frac{L_{ t}}{L}\bm{m}\cdot \bm{n}+R_3+R_5
+\frac{ 2\pi g x_{\alpha\alpha}}{L}\right)
\\
-2A\left(\left(\frac{\pi^2}{L^2}\right)\gamma\gamma_\alpha
-V_W\left(-\frac{\pi}{L}H(\gamma\theta_\alpha)+\bm{m}\cdot\bm{t}\right)+g y_\alpha\right).
\end{multline}
The new collection $R_5$ is defined as
\begin{multline}\nonumber
R_5=-\frac{2\pi}{L}\left[H,\left(\frac{2\pi}{L} V_W \theta_\alpha+\frac{2\pi}{L}\bm{m}\cdot \bm{n}\right)\right]\gamma\theta_\alpha+ \frac{2\pi}{L} V_W(V_W)_\alpha\theta_\alpha\\
+\frac{2\pi V_W}{L}(\bm{m}\cdot \bm{n})_\alpha-(\mathds{P}(\bm{m}\cdot\bm{t})
+\bm{m}\cdot\bm{t})\left(\frac{2\pi}{L} V_W \theta_\alpha+\frac{2\pi}{L}\bm{m}\cdot \bm{n}\right).
\end{multline}
We next define collections of terms $\widetilde{R}_{1},$ $\widetilde{R}_{2},$ and $R,$
\begin{multline}\label{TildeR1Def}
\widetilde{R_1}=2\widetilde{A}\Bigg(\frac{2\pi^3\theta_\alpha\gamma\gamma_\alpha}{L^3}
      -\frac{4\pi^3}{L^2}[H,(H\gamma_\alpha)]\gamma\theta_\alpha
\\
-(\mathds{P}(\bm{m}\cdot\bm{t})+\bm{m}\cdot\bm{t})\frac{2\pi^2}{L^2}H(\gamma_\alpha)-\frac{\pi L_t}{L^2}H(\gamma_\alpha)
\Bigg),
\end{multline}
\begin{multline}\nonumber
\widetilde{R_2}=
\frac{-4\pi\widetilde{A}\theta_\alpha}{L}
\left(\frac{\pi^2}{L^2}[H,\gamma] H(\gamma_\alpha)+H\left(\frac{\pi}{L} V_W \theta_\alpha\gamma
+\frac{\pi}{L}\gamma \bm{m}\cdot \bm{n}\right)\right)
\\
+\mathrm{Re}\left\{z_\alpha K[z_d]\left(\frac{\gamma z_{t\alpha}}{z_\alpha}\right)\right\}
+\mathrm{Re}\left\{\frac{z_\alpha}{2i}\left[H,\frac{1}{z_\alpha}\right]\left(\frac{\gamma z_{t\alpha}}{z_\alpha}\right)\right\}-R_1
\\
-2\widetilde{A}\left( \frac{L_{ t}}{L}\bm{m}\cdot \bm{n}+\frac{ 2\pi g x_{\alpha\alpha}}{L}+R_3+R_5\right),
\end{multline}
\begin{multline}\nonumber
R=\frac{2\pi((V_W)_\alpha\gamma)}{L}
+2A
\Bigg(\frac{\pi^2}{L^2}[H,\gamma] H(\gamma_\alpha)+H\left(\frac{\pi}{L}\gamma \bm{m}\cdot \bm{n}\right)
\\
+\mathrm{Re}\left\{z_\alpha K[z_d]\left(\frac{\gamma z_{t\alpha}}{z_\alpha}\right)\right\}
+\frac{\pi}{L}\left[H,V_W\right](\theta_\alpha\gamma)\
\\
+\mathrm{Re}\left\{\frac{z_\alpha}{2i}\left[H,\frac{1}{z_\alpha}\right]\left(\frac{\gamma z_{t\alpha}}{z_\alpha}\right)\right\}-V_W\bm{m}\cdot\bm{t}+g y_\alpha-R_1\Bigg).
\end{multline}
We are finally able to rewrite the  $\gamma_t$ equation, when $\rho>0$ and $\sigma>0,$ as
\begin{multline}\label{gammaend}
\gamma_t=-\sigma\bar{A} \partial_\alpha^4\theta -\frac{3\sigma\bar{A}\theta_\alpha^2}{2}\theta_{\alpha\alpha}+\frac{4\pi\rho\widetilde{A}}{L} V_W ^2\theta_{\alpha\alpha}+\lambda\theta_{\alpha\alpha} +\left(\frac{2\pi V_W}{L}-\frac{4A\pi^2}{L^2}\gamma\right)\gamma_\alpha
\\
-2AJ[z_d]\gamma_t
-\frac{2\rho\widetilde{A}\pi}{L}H(\gamma_{\alpha t})
-\frac{4\rho\widetilde{A}V_W\pi^2}{L^2} H(\gamma_{\alpha\alpha})
\\
+2\rho\left(S[z_d](\gamma_t)
-\frac{2\pi\widetilde{A}\theta_\alpha}{L}J[z_d]\gamma_t\right)
+\rho\widetilde{R}_1+\rho\widetilde{R}_2+R.
\end{multline}
In the case  $\rho=0$ and $\sigma=0,$ this becomes
\begin{equation}\label{gammaend00}
\gamma_t=\lambda\theta_{\alpha\alpha}+\left(\frac{2\pi V_W}{L}-\frac{4A\pi^2}{L^2}\gamma\right)\gamma_\alpha-2AJ[z_d]\gamma_t+R.
\end{equation}

Of course the $\gamma_t$ equation as in \eqref{gammaend} is not simply an evolution equation, as it is an integral
equation for $\gamma_{t}.$  To show that it is solvable, it will be helpful to introduce a brief notation for the
terms which do and do not depend on $\gamma_t$ and $\gamma_{\alpha t}.$  To this end, we introduce the following notation:
\begin{equation}\label{gammaend2}
\gamma_t=-\frac{2\rho\widetilde{A}\pi}{L}H(\gamma_{\alpha t})-T[\theta](\gamma_t)+F,
\end{equation}
where $T[\theta](\gamma_t)$ and $F$ are given by
\begin{equation}
\label{TgammaDef}
T[\theta](\gamma_t)=-2AJ[z_d]\gamma_t+2\rho\left(S[z_d](\gamma_t)-\frac{2\pi\widetilde{A}\theta_\alpha}{L}J[z_d]\gamma_t\right),
\end{equation}
\begin{multline}\label{FDef}
F=\lambda\theta_{\alpha\alpha}-\sigma\bar{A} \partial_\alpha^4\theta -\frac{3\sigma\bar{A}\theta_\alpha^2}{2}\theta_{\alpha\alpha}+\frac{4\pi\rho\widetilde{A}}{L} V_W ^2\theta_{\alpha\alpha}\\ +\left(\frac{2\pi V_W}{L}-\frac{4A\pi^2}{L^2}\gamma\right)\gamma_\alpha
-\frac{4\rho\widetilde{A}V_W\pi^2}{L^2} H(\gamma_{\alpha\alpha})+\rho\widetilde{R}_1+\rho\widetilde{R}_2+R.
\end{multline}
We will address solvability of the $\gamma_{t}$ equation in Section \ref{furtherOperatorSection} below.

\section{Preliminary estimates}\label{preliminarySection}
In what follows, we will need a number of estimates on Sobolev spaces, Hilbert transforms, and other integral operators.
We now state a number of lemmas giving these estimates.

\subsection{Estimates for the Hilbert transform and the Birkhoff-Rott integral} \label{smoothingLemmasSection}
We begin with some standard facts about Sobolev spaces, such as the following basic interpolation result.
The proof can be found in a number of sources, such as \cite{ambroseThesis}.
\begin{lemma} \label{lemma1}
Let $m\ge 0$ and $ s\ge m$ be given.  There exists $C>0$ such that for all $f\in H^s,$
\begin{align}\nonumber
\|f\|_m\le C\|f\|_s^{m/s}\|f\|_0^{1-m/s}.
\end{align}
\end{lemma}
We will also make use of the Sobolev algebra property.
\begin{lemma}\label{lemma2}
For all $s>1/2$, $H^s(\mathds{T})$ is a Banach algebra. That is, there exists $c>0$ such that for all $u,v\in H^s$,
\begin{align}\nonumber
\|uv\|_{s}\le c\|u\|_s\|v\|_s.
\end{align}
\end{lemma}
We also will need a basic composition estimate, such as can be found in \cite{taylor3}.
\begin{lemma}\label{lemma3}
If $ F$ is a smooth function and $u$ is in $H^k\cap L^\infty$, then $\|F(u)\|_k\le C (1+\|u\|_k)$. The constant C depends on $|F^{(j)}(u)|_{L^\infty}$, for $0\le j\le k$.
\end{lemma}

We will need to ensure that the interfaces we consider are non-self-intersecting, and for this purpose we follow
the approach of other works in the field such as \cite{wu2D} in using a chord-arc condition.  As in \cite{BHL}, we define $q_{1}$ to be
the first divided difference for the position of the curve,
\begin{equation}\nonumber
q_{1}[z](\alpha,\alpha',t)=\frac{z(\alpha,t)-z(\alpha',t)}{\alpha-\alpha'}.
\end{equation}
The chord-arc condition is that there exists a constant $\bar{c}>0$ such that
\begin{equation}\label{chordArc}
|q_{1}[z_{d}]|=\left|\frac{z_d(\alpha)-z_d(\alpha^\prime)}{\alpha-\alpha^\prime}\right|>\bar{c}, \quad \forall \alpha, \alpha'.
\end{equation}
When proving existence of solutions for an initial value problem as in \cite{ambroseThesis}, \cite{liu2017well},
\eqref{chordArc} is assumed at the initial time, and then the condition is shown to hold for some additional amount of time.
In the present work we are not proving that solutions exist but instead are studying the convergence of solutions
already known to exist; we may therefore assume \eqref{chordArc} for the solutions we study, on their entire interval
of existence.

Next we have bounds for the operator $K,$ which has arisen in the approximation of the Birkhoff-Rott integral by
Hilbert transforms.
\begin{lemma}\label{lemma5}
Let $n\ge 3$ be an integer. Assume $z_d\in H^n$ and that \eqref{chordArc} holds.
Then $K[z_d]:H^1 \rightarrow H^{n-1}$ and  $K[z_d]:H^0 \rightarrow H^{n-2},$ with the following estimates:
\begin{align}
\nonumber
\|K[z_d]f\|_{n-1}\le C\|f\|_1 \exp\{C_1 \|z_d\|_n\},\\
\nonumber
\|K[z_d]f\|_{n-2}\le C\|f\|_0 \exp\{C_1 \|z_d\|_n\}.
\end{align}
Furthermore,
\begin{align}\nonumber
\|K[z_d](f_\alpha)\|_{n-3}\le C\|f\|_0 \exp\{C_1 \|z_d\|_n\}.
\end{align}
\end{lemma}
We do not include the proof here, but the relevant details can be found instead in \cite{ambroseThesis}.
The proof relies on writing the kernel for $K$ in terms of divided differences as in \cite{BHL}.
In addition to the bounds of Lemma \ref{lemma5} for $K,$ we will also need a Lipschitz estimate for $K,$ to be used
when considering differences of solutions in demonstrating that a limit may be taken.

\begin{lemma}\label{lemma6}
Let $\theta$ and $\theta^\prime$ be in $ H^3$. Let $L$ and $L^\prime$ be the corresponding lengths of the associated curves $z_d$ and $z_d^\prime,$  and assume that the associated divided differences $q_{1}[z_{d}]$
and $q_{1}[z_{d}^{\prime}]$ satisfy \eqref{chordArc}. Assume there exist $ \beta>0$  such that $ L<\beta$ and $ L^\prime<\beta.$
Then there exists $C>0$ such that for any $f\in H^{3},$
\begin{align}\nonumber
\|K[z_d]f-K[z_d^\prime]f\|_i\le C\|\theta-\theta^\prime\|_i\|f\|_i
\end{align}
for any $i\in\{1,2,3\}$.
\end{lemma}
Again, we do not prove this here, but details may be found in \cite{ambroseRealThesis}, \cite{ambroseThesis} or \cite{ambroseDarcy3}.

As we have introduced the commutator of the Hilbert transform and multiplication by a smooth function, we will need
estimates for this.  Proof may be found in  \cite{liu2019sufficiently},  and similar results are contained in
\cite{ambroseThesis}, \cite{ambroseDarcy3}, and \cite{BHL}.
  \begin{lemma}\label{lemma7} Let $0\leq s\le\varpi_{1}$ be such that $s<\varpi_{1}+\varpi_{2} -1/2$.
There exists $C>0$ such that for any $\phi \in H^{\varpi_{1}}$ and $f \in H^{\varpi_{2}}$ , we have
\begin{align}\nonumber
\|[H,\phi]f\|_{H^s}\le C\|\phi\|_{H^{\varpi_{1}}}\|f\|_{H^{\varpi_{2}}}.
\end{align}
\end{lemma}

\subsection{Further operator estimates}\label{furtherOperatorSection}
We will consider solutions $(\theta,\gamma)$ lying in an open subset of $H^{s}\times H^{s-1},$ where $s$ is a sufficiently large
positive integer.  For given positive constants $\bar{d}_{1},$ $\bar{d}_{2},$ and $\bar{d}_{3},$ we let $\mathcal{O}$ be the
subset of $H^{s}\times H^{s-1}$ such that for all $(f_{1},f_{2})\in\mathcal{O},$ the following three conditions are satisfied:
\begin{align}\label{d1}
& {\|(f_1,f_2)\|^2_{H^{s-1}\times H^{s-3/2}}+\sigma\|f_1\|_{H^{s}}^2 +\rho_0\|f_2\|^2_{H^{s-1}}}<\bar{d}_{1}^{2},\\
\nonumber
&L[f_1]<\bar{d_2},\\
\nonumber
&
|q_1[f_1](\alpha,\alpha^\prime)|>\bar{d}_3, \quad \forall \alpha,\alpha^\prime. 
\end{align}
We notice that when $\sigma=0$ and $\rho_0=0$, the inequality \eqref{d1}
becomes simply $\|(f_1,f_2)\|_{H^{s-1}\times H^{s-3/2}}<\bar{d}_1$.

We now consider $(\theta,\gamma)\in \mathcal{O},$ and we address the solvability of our integral equation \eqref{gammaend2}
for $\gamma_{t}.$  As the relevant operators involve involve $z_d$ and $z_\alpha$, we will first give
estimates for $z_d$ and $z_\alpha$ in terms of $\theta$.

\begin{lemma}\label{lemmazazd}
Let $(\theta,\gamma)\in \mathcal{O}$ be given. Then, the following estimates are satisfied:
\begin{align}
&\|z_\alpha\|_s\le C(1+\|\theta\|_s), \label{za}\\
&\|z_d\|_{s+1}\le C(1+\|\theta\|_s). \label{zd}
\end{align}
\end{lemma}
\begin{proof}
The inequality \eqref{za}  follows immediately from \eqref{xya}
together with the standard composition estimate (Lemma \ref{lemma3}) and  the fact that the definition of $\mathcal{O}$ includes a bound on the length.   Since $z_d$ is defined in \eqref{zdd} by an integral of $x_\alpha$ and $y_\alpha,$ the estimate \eqref{zd} follows.
\end{proof}

Recall $T[\theta]\gamma_t$, $J[z_d](\gamma_t)$ and $S[z_d](\gamma_t)$ are defined by \eqref{JzDef}, \eqref{Sdef}, and \eqref{TgammaDef}.
  Generally for any $f\in H^0$, we have the definitions
\begin{align}\nonumber
J[z_d]f=\mathrm{Re}\left\{z_\alpha K[z_d]f+\frac{z_\alpha}{2i}
\left[H,\frac{1}{z_\alpha}\right]f\right\}
\end{align}
\begin{equation}\nonumber
S[z_d](f)=\mathrm{Re}\left\{\frac{i z_\alpha^2}{s_\alpha}K[z_d]\left(\left(\frac{f}{z_\alpha}
\right)_\alpha\right)\right\}
+\mathrm{Re}\left\{\frac{z_\alpha^2}{2s_\alpha}
\left[H,\frac{1}{z_\alpha^2}\right]\left(z_\alpha\left(\frac{f}{z_\alpha}\right)_\alpha\right)\right\},
\end{equation}
and
\begin{equation}
\label{TDef}
T[\theta](f)=-2AJ[z_d]f+2\rho\left(S[z_d](f)-\frac{2\pi\widetilde{A}\theta_\alpha}{L}J[z_d]f\right).
\end{equation}
Moreover, we have the following estimates for $J[z_d]$ and $S[z_d]$.
 \begin{lemma}\label{Tbound}
 Let $(\theta,\gamma)\in \mathcal{O}$ be given.
 Then $J[z_d]$ and $S[z_d],$ which are defined by \eqref{JzDef} and \eqref{Sdef}, are bounded operators from $H^0$ to $H^{s-1}$, and there exist positive constants
 $C_1$ and $C_2$ such that for any $f\in H^0$,
\begin{align}\nonumber
\|J[z_d]f\|_{s-1}\le C_1\exp\{C_2\|\theta\|_s\}\|f\|_0,\\
\|S[z_d](f)\|_{s-1}\le C_1\exp\{C_2\|\theta\|_s\}\|f\|_{1}.\nonumber
\end{align}
Moreover, the following Lipschitz estimates hold, for any $f\in H^{3}:$
\begin{align}\label{Lip1}
\|J[z_d]f-J[z_d^\prime]f\|_{1}\le C_1\|\theta-\theta^\prime\|_1\|f\|_1,\\
\|S[z_d](f)-S[z_d^\prime](f)\|_{1}\le C_1\|\theta-\theta^\prime\|_1\|f\|_{1}.\label{Lip2}
\end{align}
 \end{lemma}
 \begin{proof}
 The operators contain $K[z_d]$ and Hilbert commutators. The following estimates take them into account separately.
By Lemma \ref{lemma5}, for $n\ge3,$ we have the following estimates:
\begin{align*}
\|K[z_d]f\|_{s-1}\le  C_1\|f\|_{0}\exp\{C_2\|z_d\|_{s+1}\}
\end{align*}
and
\begin{align*}
&\left\|K[z_d]\left(\left(\frac{f}{z_\alpha}\right)_{\alpha }\right)\right\|_{s-1}
\le C_1\left\|\left(\frac{f}{z_\alpha}\right)_\alpha\right\|_{0}\exp\{C_2\|z_d\|_{s+1}\}\le C_1\|f\|_1\exp\{C_2\|z_d\|_{s+1}\}.
\end{align*}
By Lemma \ref{lemma7}, we have the following:
\begin{align*}
&\left\|\left[H,\frac{1}{z_\alpha^2}\right]f\right\|_{s-1}\le
C_1\left\|\frac{1}{z^2_\alpha}\right\|_{s}\|f\|_0\le C_1\|f\|_0\exp\{C_2\|z_\alpha\|_s\}, \\
\end{align*}
and
\begin{align*}
&\left\|\left[H,\frac{1}{z_\alpha^2}\right]\left(z_\alpha\left(\frac{f}{z_\alpha}\right)_{\alpha }\right)\right\|_{s-1}
\le C_1\left\|\frac{1}{z_\alpha^2}\right\|_{s}
\left\|\left(z_\alpha\left(\frac{f}{z_\alpha}\right)_{\alpha }\right)\right\|_{0}\le C_1\|f\|_1\exp\{C_2\|z_\alpha\|_s\}.
\end{align*}
For the Lipschitz estimates, we first consider $J[z_d]f-J[z_d^\prime]f$ and $S[z_d]f-S[z_d^\prime]f$. Since these both involve  
$K[z_d]$ and Hilbert commutators, here we only
work with $J[z_d]f-J[z_d^\prime]f$ as an example.  Adding and subtracting, we have
\begin{multline}\nonumber
J[z_d]f-J[z_d^\prime]f
\\
=\mathrm{Re}\left\{z_\alpha K[z_d]f+\frac{z_\alpha}{2i}
\left[H,\frac{1}{z_\alpha}\right]f\right\}-\mathrm{Re}\left\{z_\alpha^\prime K[z_d^\prime]f+\frac{z_\alpha^\prime}{2i}
\left[H,\frac{1}{z_\alpha^\prime}\right]f\right\}
\\
=\mathrm{Re}\left\{(z_\alpha-z_\alpha^\prime) K[z_d]f+\frac{z_\alpha-z_\alpha^\prime}{2i}
\left[H,\frac{1}{z_\alpha}\right]f\right\}
\\
+\mathrm{Re}\left\{z_\alpha^\prime (K[z_d]-K[z_d^\prime])f
+\frac{z_\alpha^\prime}{2i}
\left[H,\frac{1}{z_\alpha}-\frac{1}{z_\alpha^\prime}\right]f\right\}.
\end{multline}
Using the estimates for $K[z_d]$ and Hilbert commutators of Lemma \ref{lemma5} and Lemma \ref{lemma7}, and the Lipschitz estimate for $K[z_d]$ of Lemma \ref{lemma6},
we have
\begin{align}\nonumber
\|J[z_d]f-J[z_d^\prime]f\|_{1}\le C_1\|\theta-\theta^\prime\|_1\|f\|_{1}.
\end{align}
(Note that the constants $C_{1}$ and $C_{2}$ may depend upon the constants $\bar{d}_{1},$ $\bar{d}_{2},$
and $\bar{d}_{3}$ which define our open set, but this causes no problem.)
\end{proof}

The article \cite{liu2017well} has proved well-posedness of our initial value problem under
this assumption that the integral equation for $\gamma_{t}$ is solvable.  We again make such an assumption, but now
need the assumption to hold uniformly with respect to $\rho_{0}$ and $\sigma.$
\begin{assumption}\label{solvabilityAssumption}
There exist $c^{*}>0,$ $\rho^{*}>0,$ and $\sigma^{*}>0$ such that for all $\rho_{0}\in[0,\rho^{*}],$
for all $\sigma\in(0,\sigma^{*}],$ and for all $(\theta,\gamma)\in\mathcal{O},$
the operator $I+\frac{2\rho\widetilde{A}}{L}H\partial_{\alpha}+T[\theta]$ is invertible as an operator from
$H^{0}$ to $H^{0},$ with the bound
\begin{equation}\nonumber
\left\|\left(I+\frac{2\rho\widetilde{A}}{L}H\partial_{\alpha}+T[\theta]\right)^{-1}\right\|_{H^{0}\rightarrow H^{0}}\leq c^{*}.
\end{equation}
\end{assumption}
Of course, we should show that this assumption can be satisfied.  To develop an example of when Assumption \ref{solvabilityAssumption}
may be satisfied, we begin by noting the formula
\begin{equation}\nonumber
(I+B_{1}+B_{2})^{-1}=(I+(I+B_{1})^{-1}B_{2})^{-1}(I+B_{1})^{-1}.
\end{equation}
We further have the estimate
\begin{equation}\nonumber
\left\|(I+B_{1})^{-1}B_{2}\right\|_{H^{0}\rightarrow H^{0}}\leq \left\|(I+B_{1})^{-1}\right\|_{H^{0}\rightarrow H^{0}}\|B_{2}\|_{H^{0}\rightarrow H^{0}}.
\end{equation}
If  $\left\|(I+B_{1})^{-1}\right\|_{H^{0}\rightarrow H^{0}}\leq 1$ and $\|B_{2}\|_{H^{0}\rightarrow H^{0}}<1,$ then we see that
$I+B_{1}+B_{2}$ is invertible with bounded inverse.
Here, the operator $B_{1}$ is $B_{1}=\frac{2\rho\widetilde{A}}{L}H\partial_{\alpha},$ and we see immediately from the Fourier symbol of $H\partial_{\alpha}$
that
\begin{equation}\nonumber
\left\|(I+B_{1})^{-1}\right\|_{H_{0}\rightarrow H_{0}}\leq 1.
\end{equation}
The operator $B_{2}$ is $B_{2}=T[\theta],$ and so we see that Assumption \ref{solvabilityAssumption} is satisfied if
\begin{equation}\label{TAssumption}
\|T[\theta]\|_{H_{0}\rightarrow H^{0}}<1.
\end{equation}
From the definition of $T$ and estimates of Lemma \ref{Tbound} for the operators $J[z_{d}]$ and $S[z_{d}],$ we see that \eqref{TAssumption} holds
if $\rho_{0}$ is small and if $\rho_{1}-\rho_{2}$ is also small.  Note that the assumption that $\rho_{0}$ is small is entirely natural here since
we are sending $\rho_{0}$ to zero.  Of course Assumption \ref{solvabilityAssumption} may be satisfied more generally, but
we can guarantee that it holds in this case of small $\rho_{1}-\rho_{2}.$

We further remark on another case in which Assumption \ref{solvabilityAssumption} is known to be satisfied.  A number of works on hydroelastic waves,
such as \cite{ambroseSiegel-hydroelastic}, consider the massless case, i.e. the case in which $\rho_{0}=0.$  The results of the present work
also apply in this case, so that the zero-bending limit of the massless case may be taken.  In this case Assumption \ref{solvabilityAssumption}
is known to be satisfied.  In this case the operator to be inverted is simply $I-2AJ[z_{d}],$ and it was established in \cite{BMO} that the spectral radius
of $2AJ[z_{d}]$ is less than $1,$ and thus the operator is invertible and the assumption is satisfied.

\begin{lemma}\label{priori}
Let $(\theta,\gamma)\in \mathcal{O}$ be given, such that $\theta$ satisfies  $\llangle \sin(\theta)\rrangle=0$ . Then, the following estimates are satisfied:
\begin{align}
&\|\bm{t}\|_i\le C(1+\|\theta\|_i),\quad i=s-1,s,\label{t}\\
&\|\bm{n}\|_i\le C(1+\|\theta\|_i),\quad i=s-1,s,\label{n}\\
&\|\bm{m}\|_i\le C_1\|\gamma\|_{s-3/2}\exp(C_2\|\theta\|_i),\quad s-3/2\le i\le s,\label{m}\\
&\|\bm{W}\|_{s-3/2}\le C_1\|\gamma\|_{s-3/2}\exp(C_2\|\theta\|_{s-1}),\label{Wpri}\\
&\|U\|_{s-3/2}\le C_1\|\gamma\|_{s-3/2}\exp(C_2\|\theta\|_{s-1}), \label{U}\\
&|L_t|\le C_1\|\gamma\|_{s-3/2}\exp(C_2\|\theta\|_{s-1}), \label{Lt}\\
&\|V_W\|_{s-1}\le C_1\|\gamma\|_{s-3/2}\exp(C_2\|\theta\|_{s-1}), \label{VW}\\
&\|V\|_{s-1}\le C_1\|\gamma\|_{s-3/2}\exp(C_2\|\theta\|_{s-1}), \label{V}\\
&\|V\|_{s-3/2}\le C_1\|\gamma\|_{s-3/2}\exp(C_2\|\theta\|_{s-1}), \label{V32}\\
&\|z_t\|_{s-3/2}\le C_1\|\gamma\|_{s-3/2}\exp(C_2\|\theta\|_{s-1}),\label{zt}\\
&\|R_1\|_{s-3/2}\le C_1\|\gamma\|_{s-3/2}\exp(C_2\|\theta\|_{s-1}),\label{R1}\\
&\|R_3\|_{s-2}\le C_1\|\gamma\|_{s-3/2}\exp(C_2\|\theta\|_{s-1}),\label{R3}\\
&\|R_5\|_{s-2}\le C_1\|\gamma\|_{s-3/2}\exp(C_2\|\theta\|_{s-1}),\label{R5}\\
&\|\widetilde{R}_1\|_{s-2}\le C_1\|\gamma\|_{s-1}\exp(C_2\|\theta\|_{s-1}),\label{TildeR1}\\
&\|\widetilde{R}_2\|_{s-2}\le C_1\|\gamma\|_{s-3/2}\exp(C_2\|\theta\|_{s-1}),\label{TildeR2}\\
&\|R\|_{s-3/2}\le C_1\|\gamma\|_{s-3/2}\exp(C_2\|\theta\|_{s-1}).\label{R}
\end{align}
\end{lemma}
\begin{proof}
We will use Lemma \ref{lemmazazd} in the following when we bound $z_\alpha$ and $z_d$.
 The estimates  \eqref{t} and \eqref{n} follow immediately from  \eqref{tnDef}, together with the standard composition estimate, Lemma \ref{lemma3}.
By the definition of $\bm{m}$ and Lemma \ref{lemma2}, Lemma \ref{lemma5}, and Lemma \ref{lemma7},
we have for $s-3/2\le i\le s,$
\begin{align*}
\|\bm{m}\|_i&\le\|z_\alpha \|_i
\left\| K[z_d]\left(\frac{\gamma}{z_\alpha}\right)_\alpha\right\|_i
+\left\|\frac{z_\alpha}{2i}\right\|_i
\left\|\left[H,\frac{1}{(z_\alpha)^2}\right]
\left(z_\alpha\left(\frac{\gamma}{z_\alpha}\right)_\alpha\right)\right\|_i\\
&\le C_1\|z_\alpha \|_i
\left\|\left(\frac{\gamma}{z_\alpha}\right)_\alpha\right\|_1
\exp\{C_2\|z_d\|_{i+1}\}
+\left\|\frac{z_\alpha}{2i}\right\|_i
\left\|\frac{1}{(z_\alpha)^2}\right\|_i
\left\|z_\alpha\left(\frac{\gamma}{z_\alpha}\right)_\alpha\right\|_{1}\\
&\le C_1\|\gamma\|_{s-3/2}\exp(C_2\|\theta\|_i).
\end{align*}
By the definition of $\bm{W}$, and adding and subtracting, we rewrite it as
\begin{align}\nonumber
\mathcal{C}(\bm{W})^*=K[z_d]\gamma
+\frac{1}{2i}H\left(\frac{\gamma}{z_\alpha}\right).
\end{align}
Lemma \ref{lemma5} provides a bound for $K[z_{d}]:$
\begin{multline}\label{WaPriori}
\|\bm{W}\|_{s-3/2}\le \|K[z_d]\gamma\|_{s-3/2}
+\left\|\frac{1}{2i}H\left(\frac{\gamma}{z_\alpha}\right)\right\|_{s-3/2}\\
\le C_1\|\gamma\|_{s-3/2}\exp(C_2\|\theta\|_{s-1}).
\end{multline}
The estimate \eqref{U} immediately follows from \eqref{UDef} and the bounds  \eqref{n} and \eqref{WaPriori}.
The estimate \eqref{Lt} follows from \eqref{LtDef}, \eqref{U} and the Schwartz inequality.
The estimate on $V_W$ readily follows from its definition \eqref{V_Wep} and the bound on $\bm{m}$.
To estimate $V$, we first consider $V_0(t)$. By the definition \eqref{V0Def}, $\left|V_0(t)\right|$ is bounded from \eqref{t} and \eqref{Wpri} and the Schwartz inequality.
 Then the estimate on $V$  is straightforward  from  its definition \eqref{VDef} and the bound on $U.$
The estimate \eqref{zt} is obtained easily from the definition $z_t=\mathcal{C}(U\bm{n}+V\bm{t})$ and previous bounds.

To get the estimate on $R_1$ and $R$, we first obtain bounds on terms which involve the operator $K[z_d]$ and various Hilbert commutators. These are all smoothing enough for sufficiently large $s$ by Lemma \ref{lemma5} and Lemma \ref{lemma7}. So, all the terms comprising $R_1$ and $R$ are bounded in $H^{s-3/2}$ from Lemma \ref{lemma2} and all previous estimates. This demonstrates the inequalities \eqref{R1} and \eqref{R}.
To get the estimate on $R_3,$ $R_5,$ and $\tilde{R}_2$, we again first obtain bounds on terms which involve the operator $K[z_d]$ and various Hilbert commutators, which again implies  \eqref{R3}, \eqref{R5}, and \eqref{TildeR2}.
We have defined $\tilde{R}_1$ in \eqref{TildeR1Def}, and it contains as
the highest derivative terms $\gamma_\alpha$ and $\theta_{\alpha}$.
So the inequality \eqref{TildeR1} follows from Lemma \ref{lemma2} and all previous estimates.
\end{proof}

\section{Uniform Time of Existence}\label{uniformSection}
In order to take the limit as $\sigma$ and $\rho_{0}$ vanish,
we must demonstrate that solutions exist on a time interval independent of these parameters.
We therefore revisit the energy estimates of the authors in \cite{liu2017well} to this end.  Throughout this section,
we consider the parameters defining $\mathcal{O}$ as well as
$\rho_{1}$ and $\rho_{2}$ to be fixed.

\begin{theorem}\label{uniform_time}
 Let $(\theta_0,\gamma_0)\in \mathcal{O}$ be given, with $\llangle \sin(\theta_0)\rrangle=0$.
Assume that Assumption \ref{solvabilityAssumption} holds.  Then
there exists $T>0$ such that for any $\rho_{0}\in[0,\rho^{*}]$ and $\sigma\in(0,\sigma^{*}],$ 
the  solution of Cauchy problem \eqref{evolution of theta}, \eqref{gammaend} with initial conditions $(\theta_0,\gamma_0)$ exists and is
 $(\theta^{\sigma,\rho_{0}},\gamma^{\sigma,\rho_{0}})\in C([0,T],\mathcal{O})$. Furthermore there exist constants $c_1\in(0,\infty),c_2\in(0,1)$ and $c_3\in(0,\infty)$ , depending only on $ s, \bar{d_1},\bar{d_2},\bar{d}_3,\|\theta_0\|_s$ and $\|\gamma_0\|_{s-1}$, such that
 \begin{equation}\nonumber
 \sigma\|\theta^{\sigma,\rho_{0}}\|_s^2+\|\theta^{\sigma,\rho_{0}}\|_{s-1}^2+\|\gamma^{\sigma,\rho_{0}}\|_{s-3/2}^2+\rho_0\|\gamma^{\sigma,\rho_{0}}\|_{s-1}^2\le-c_1\ln(c_2-c_3t).
 \end{equation}
\end{theorem}
\begin{proof} For simplicity, in the present proof we will drop the superscripts, and refer to $(\theta^{\sigma,\rho_{0}},\gamma^{\sigma,\rho_{0}})$ as $(\theta,\gamma).$
We define an energy functional $E$  as
\begin{align}
\nonumber&E(t)=E_0(t)+\sigma E_1(t)+E_{2}(t)+E_{3}(t)+\rho E_{4}(t),
\end{align}
where  the definitions of $E_0$, $E_{1},$ $E_{2},$ $E_{3},$ and $E_{4}$ are
\begin{align}
\nonumber&E_0(t)=\frac12\int_0^{2\pi}(\theta)^2+(\gamma)^2\ d\alpha,\\
\nonumber&E_1(t)=\frac{L^2 \bar{A}}{4\pi^2 }\int_0^{2\pi}(\partial_\alpha^s\theta)^2\ d\alpha,\\
\nonumber&E_2(t)=\frac12\int_0^{2\pi}(\partial^{s-2}_\alpha \gamma)( H\partial_\alpha^{s-1} \gamma)\ d\alpha,\\
\nonumber&E_{3}(t)=\left(\frac{\tau L}{\pi(\rho_1+\rho_2)}+\frac{\sigma \bar{A} L^2}{8\pi^2}\right)\int_0^{2\pi}(\partial_\alpha^{s-1}\theta)^2\ d\alpha,\\
\nonumber&E_{4}(t)=\frac{\pi \widetilde{A}}{L}\int_0^{2\pi}(H\partial_\alpha^{s-1} \gamma)^2\ d\alpha.
\end{align}
Recalling that $\rho=\rho_{0}/s_{\alpha},$ we see that the terms $\sigma E_{1}$ and $\rho E_{4}$ are tracking leading-order contributions related to
the bending and mass of the filament.
We have that $\sigma\bar{A}\ge0$ and $\widetilde{A}\ge0$  and $L$ is non-negative as well as bounded above and below.
Therefore we see that the energy is non-negative for all $t.$

We take the time derivative of $E_0$:
\[
\frac{dE_0}{dt}=\int_0^{2\pi} \theta\theta_t+\gamma\gamma_t \ d\alpha.
\]
Since the Sobolev index 
$s$ is sufficiently largely, it is immediate, from the evolution equations \eqref{evolution of theta} and \eqref{gammaend}, and related equations, as well as the estimates from Lemma \ref{priori}, that the following inequality holds:
\begin{equation}\nonumber
\frac{dE_0}{dt}\le C_1 \exp(C_2 E).
\end{equation}

We next take the time derivative of $E_1$:
\begin{equation}\label{de1}
\frac{dE_1}{dt}=\frac{(L^2 \bar{A})_t}{4\pi^2 }\int_0^{2\pi}(\partial_\alpha^s\theta)^2\ d\alpha+
\frac{L^2 \bar{A}}{2\pi^2 }\int_0^{2\pi}(\partial_\alpha^s\theta)
(\partial_\alpha^s\theta_t)\ d\alpha.
\end{equation}
 To proceed with estimating \eqref{de1}, we will focus on the second term on the right-hand side.  Applying $\partial_\alpha^s$ to
\eqref{evolution of theta}, we get
\begin{align}\label{thetats}
\partial_\alpha^s\theta_t
= \frac{2\pi^2}{L^2} H(\partial_\alpha^{s+1}\gamma)
+\frac{2\pi V_W}{L}(\partial_\alpha^{s+1}\theta)+\frac{2\pi\theta_\alpha}{L}(\partial_\alpha^{s}V_W)
+\Phi_1,
\end{align}
where $\Phi_1$ is given by the formula
\[
\Phi_1= \frac{2\pi}{L}\sum_{j=1}^{s-1}\binom{s}{j}
  (\partial_{\alpha}^jV_W)(\partial_\alpha^{s+1-j}\theta)
 +\frac{2\pi}{L}\partial_\alpha^s(\bm{m}\cdot \bm{n}).
 \]
We plug \eqref{thetats} into \eqref{de1}, and multiply by $\sigma,$ finding
\begin{multline}\label{de12}
\frac{d(\sigma E_1)}{dt}=\frac{(L^2 \sigma\bar{A})_t}{4\pi^2 }\int_0^{2\pi}(\partial_\alpha^s\theta)^2
\ d\alpha+
 \sigma\bar{A}\int_0^{2\pi}(\partial_\alpha^s\theta)\left( H\partial_\alpha^{s+1}\gamma\right)\ d\alpha\\
 +\frac{L \sigma\bar{A}}{\pi }\int_0^{2\pi}( \partial_\alpha^s\theta)
V_W(\partial_\alpha^{s+1}\theta)
\ d\alpha+\frac{L \sigma\bar{A}}{\pi }\int_0^{2\pi}( \partial_\alpha^s\theta)
(\partial_\alpha^{s}V_W)\theta_\alpha
\ d\alpha
\\
+\frac{L^{2} \sigma\bar{A}}{2\pi^{2}}
\int_0^{2\pi}(\partial_\alpha^s\theta)\Phi_1  \ d\alpha.
\end{multline}
 For the last term on the right-hand side of \eqref{de12}, all of the summands comprising $\Phi$
 involve at most $s$ derivatives of $\sigma\theta$ and $\sigma V_W$. Therefore, the estimates of Lemma \ref{priori} immediately imply  that $ \left|\frac{(L)^{2} \sigma\bar{A}}{2\pi^{2}}
\int_0^{2\pi}\partial_\alpha^s\theta\Phi_1  d\alpha\right|\le C_1\exp(C_2 E)$.
For the fourth term on the right-hand side of \eqref{de12}, applying $\partial_\alpha^s$ to $V_W$ using \eqref{V_Wep}, and extracting the leading-order term, we have
\begin{multline}\nonumber
\partial_\alpha^{s}V_W=\frac{\pi}{L}\theta_\alpha H(\partial_{\alpha}^{s-1}\gamma)+\frac{\pi}{L}[H,\theta_\alpha]\partial_{\alpha}^{s-1}\gamma
\\
+\sum_{j=0}^{s-2}\binom{s-1}{j}H((\partial_\alpha^j\gamma)(\partial_\alpha^{s-j}\theta))-\partial_\alpha^{s-1}(\bm{m}\cdot\bm{t}).
\end{multline}
Using this, and further rearranging terms, we have
\begin{multline}\label{de13}
\frac{d(\sigma E_1)}{dt}=
 \sigma\bar{A}\int_0^{2\pi}(\partial_\alpha^s\theta)\left( H\partial_\alpha^{s+1}\gamma\right)\ d\alpha
 \\
 +\sigma\bar{A}\int_0^{2\pi}(\partial_\alpha^s\theta)\theta_\alpha^2\left( H\partial_\alpha^{s-1}\gamma\right)\ d\alpha+\Psi_1,
\end{multline}
where $\Psi_1$ is given by
\begin{multline}\nonumber
\Psi_1=\frac{(L^2 \sigma\bar{A})_t}{4\pi^2 }\int_0^{2\pi}(\partial_\alpha^s\theta)^2\ d\alpha
-\frac{L \sigma\bar{A}}{2\pi }\int_0^{2\pi}( \partial_\alpha^s\theta )^2 (\partial_\alpha V_W)\ d\alpha
\\
+\frac{L^{2} \sigma\bar{A}}{2\pi^{2}}\int_0^{2\pi}(\partial_\alpha^s\theta)\Phi_1 \ d\alpha
+\sigma\bar{A}\int_{0}^{2\pi}(\partial_{\alpha}^{s}\theta)[H,\theta_\alpha]\partial_{\alpha}^{s-1}\gamma\ d\alpha
\\
+\frac{L\sigma\bar{A}}{\pi}\sum_{j=0}^{s-2}\binom{s-1}{j}\int_{0}^{2\pi}(\partial_{\alpha}^{s}\theta)
H((\partial_\alpha^j\gamma)(\partial_\alpha^{s-j}\theta))\ d\alpha
\\
-\frac{L\sigma\bar{A}}{\pi}\int_{0}^{2\pi}(\partial_{\alpha}^{s}\theta)(\partial_\alpha^{s-1}(\bm{m}\cdot\bm{t}))\ d\alpha.
\end{multline}
We have an immediate estimate for $\Psi_{1},$ namely
\begin{align}\nonumber
 |\Psi_1|\le C_1\exp(C_2 E).
\end{align}
Arguing in the same manner, we are also able to bound the growth of $E_{3}$ as
\begin{multline}\label{dE1prime}
\frac{dE_{3}}{dt}\le\lambda\int_0^{2\pi}(\partial_\alpha^{s-1}\theta)\left( H\partial_\alpha^{s}\gamma\right)\ d\alpha
\\
-\frac{\sigma \bar{A}}{2}\int_0^{2\pi}(\partial_\alpha^{s}\theta)\left( H\partial_\alpha^{s-1}\gamma\right)\ d\alpha+C_1\exp(C_2 E),
\end{multline}
where we recall the definition $\lambda=\displaystyle\frac{4\tau\pi}{L(\rho_{1}+\rho_{2})}.$

Before we compute $ \frac{d E_2}{dt},$ we apply $\partial_{\alpha}^{s-2}$ to \eqref{gammaend}, finding
\begin{multline}\nonumber
\partial_\alpha^{s-2}\gamma_t=-\sigma\bar{A} \partial_\alpha^{s+2}\theta -\frac{3\sigma\bar{A}\theta_\alpha^2}{2}\partial_\alpha^s\theta+\frac{4\pi\rho\widetilde{A} V_W ^2}{L}\partial_\alpha^s\theta+\lambda\partial_\alpha^s\theta
\\
+\left(\frac{2\pi V_W}{L}-\frac{4A\pi^2}{L^2}\gamma\right)\partial_\alpha^{s-1}\gamma-\frac{2\rho\widetilde{A}\pi}{L}H(\partial_\alpha^{s-1}\gamma_t)
-\frac{4\rho\widetilde{A}V_W\pi^2}{L^2} H(\partial_\alpha^s\gamma)
\\
-\frac{4(s-2)\rho\widetilde{A}(\partial_\alpha V_W)\pi^2}{L^2} H(\partial_\alpha^{s-1}\gamma)
+\rho\partial^{s-2}_\alpha\widetilde{R}_1+\Phi_2+\rho\tilde{\Phi}_2.
\end{multline}
Here, $\Phi_2$  and $\tilde{\Phi}_2$ are given by the formulas
\begin{multline}\nonumber
\Phi_2=\partial_\alpha^{s-2}\left(-2AJ[z_d]\gamma_t+R\right)
\\
+\sum_{j=0}^{s-3}\binom{s-2}{j}(\partial_\alpha^{j+2}\theta)\partial_{\alpha}^{s-2-j}\left(\frac{-3\sigma\bar{A}\theta_\alpha^2}{2}+\frac{4\pi\rho\widetilde{A}V_W^2}{L}\right)
\\
+\sum_{j=0}^{s-3}\binom{s-2}{j}(\partial_\alpha^{j+1}\gamma)
\partial_{\alpha}^{s-2-j}\left(\frac{2\pi V_W}{L}-\frac{4A\pi^2}{L^2}\gamma\right),
\end{multline}
\begin{multline}\nonumber
\rho\tilde{\Phi}_2=\partial_\alpha^{s-2}\left(2\rho\left(S[z_{d}](\gamma_t)-\frac{2\pi\widetilde{A}\theta_\alpha}{L}J[z_d]\gamma_t\right)+\rho\widetilde{R}_2\right)
\\
-\sum_{j=0}^{s-4}\binom{s-2}{j}\frac{4\rho\widetilde{A}(\partial^{s-2-j}_\alpha V_W)\pi^2}{L^2} H(\partial_\alpha^{j+2}\gamma).
\end{multline}
We remark that the distinction between these two groupings is that $\tilde{\Phi}_{2}$ is for terms proportional to $\rho.$

We notice that it is straightfoward to conclude that $\gamma_t$ is bounded in $H^0$ when $s\ge 3.$  Then we see that
the collection of terms
$\Phi_2$ involves the smoothing operator $J[z_d]$ applied to $\gamma_{t},$ as well as $R,$ and other terms involving
at most $s-1$ derivatives of $\theta$ and at most $s-2$ derivatives of $\gamma.$
We conclude the bound $\|\Phi_2\|_{1/2}\le C_1\exp(C_2 E)$.
Similarly, the collection of terms $\tilde{\Phi}_2$ involves the smoothing operators $S$ and $\mathcal{J}$ applied to $\gamma_{t},$
the term $\widetilde{R}_{2},$
as well as terms involving at most $s-1$ derivatives of $\theta$ and at most $s-2$ derivatives of $\gamma$,
so we may conclude the estimate $\|\tilde{\Phi}_2\|_{0}\le \rho C_1\exp(C_2 E)$.
Note that we have used Lemma \ref{priori} to estimate the quantities $R$ and $\widetilde{R}_{2},$ and we have
used Lemma \ref{Tbound} to estimate the operators $S$ and $\mathcal{J}.$

We are now in a position to compute $dE_{2}/dt,$ expanding as follows:
\begin{multline}\label{de2}
\frac{dE_{2}}{dt}=
-\sigma\bar{A} \int_0^{2\pi} \left(H\partial_\alpha^{s-1} \gamma\right)\partial_\alpha^{s+2}\theta \ d\alpha
-\int_0^{2\pi} \left(H\partial_\alpha^{s-1} \gamma\right)\frac{3\sigma\bar{A}\theta_\alpha^2}{2}\partial_\alpha^s\theta\ d\alpha
\\
+\int_0^{2\pi} \frac{4\pi\rho\widetilde{A}V_W^2}{L}\left(H\partial_\alpha^{s-1} \gamma\right)\partial_\alpha^s\theta\ d\alpha
+\int_0^{2\pi} \left(H\partial_\alpha^{s-1} \gamma\right)\lambda\partial_\alpha^s\theta\ d\alpha
\\
+\int_0^{2\pi} \left(H\partial_\alpha^{s-1} \gamma\right)\left(\frac{2\pi V_W}{L}-\frac{4A\pi^2}{L^2}\gamma\right)\partial_\alpha^{s-1}\gamma\ d\alpha
\\
-\frac{2\rho\widetilde{A}\pi}{L}\int_0^{2\pi} \left(H\partial_\alpha^{s-1} \gamma\right)
H\partial_\alpha^{s-1}\gamma_{t}\ d\alpha
-\frac{4\rho\pi^{2}\widetilde{A}}{L^2}\int_0^{2\pi}\left( H\partial_\alpha^{s-1} \gamma\right)
V_W  H\partial_\alpha^{s} \gamma\ d\alpha
\\
-\frac{4(s-2)\rho\pi^{2}\widetilde{A}}{L^2}
\int_0^{2\pi}\left( H\partial_\alpha^{s-1} \gamma\right)^{2}
(\partial_\alpha V_W)\ d\alpha
\\
+\rho\int_0^{2\pi}(H\partial_\alpha^{s-1}\gamma)\tilde{R}_1 \ d\alpha
+\int_0^{2\pi}(\Lambda^{1/2}\partial_\alpha^{s-2}\gamma)\Lambda^{1/2}\Phi_2 \ d\alpha
+\rho\int_0^{2\pi}(H\partial_\alpha^{s-1}\gamma)\tilde{\Phi}_2 \ d\alpha.
\end{multline}
We will deal with the terms on the right-hand side of \eqref{de2} one by one.
We integrate by parts  in the first and the fourth terms on the right-hand side of \eqref{de2}, finding
\[
-\sigma\bar{A} \int_0^{2\pi} \left(H\partial_\alpha^{s-1} \gamma\right)
\partial_\alpha^{s+2}\theta  \ d\alpha
=-\sigma\bar{A} \int_0^{2\pi}
\left( H\partial_\alpha^{s+1} \gamma\right) \partial_\alpha^{s}\theta  \ d\alpha,
\]
\[
\int_0^{2\pi} \left(H\partial_\alpha^{s-1} \gamma\right)\lambda\partial_\alpha^s\theta d\alpha=-\int_0^{2\pi} \left(H\partial_\alpha^{s} \gamma\right)\lambda\partial_\alpha^{s-1}\theta d\alpha.
\]
Notice that these expressions are the same as the first term on the right-hand side of \eqref{de13} and the first term on the right-hand side of \eqref{dE1prime} with the opposite signs, respectively.

Using the skew-adjointedness property of the Hilbert transform, and then introducing a commutator,
we rewrite the fifth term on the right-hand side of \eqref{de2}:
\begin{multline}\nonumber
\int_0^{2\pi} \left(H\partial_\alpha^{s-1} \gamma\right)\left(\frac{2\pi V_W}{L}-\frac{4A\pi^2}{L^2}\gamma\right)\partial_\alpha^{s-1}\gamma \ d\alpha
\\
=-\int_0^{2\pi} \left(\partial_\alpha^{s-1} \gamma\right)H\left[\left(\frac{2\pi V_W}{L}-\frac{4A\pi^2}{L^2}\gamma\right)\partial_\alpha^{s-1}\gamma\right] \ d\alpha
\\
=-\int_0^{2\pi} \left(\partial_\alpha^{s-1} \gamma\right)\left[H,\left(\frac{2\pi V_W}{L}-\frac{4A\pi^2}{L^2}\gamma\right)\right]\partial_\alpha^{s-1}\gamma \ d\alpha
\\
-\int_0^{2\pi} \left(\partial_\alpha^{s-1} \gamma\right)\left(\frac{2\pi V_W}{L}-\frac{4A\pi^2}{L^2}\gamma\right)H\partial_\alpha^{s-1}\gamma \ d\alpha.
\end{multline}
Since the second term on the right-hand side is the same as the left-hand side but with opposite sign, these combine.
By Lemma \ref{lemma7}, we may then bound this as
\begin{multline}\nonumber
\left|-\frac{1}{2}\int_0^{2\pi} \left(\partial_\alpha^{s-1} \gamma\right)\left[H,\left(\frac{2\pi V_W}{L}-\frac{4A\pi^2}{L^2}\gamma\right)\right]\partial_\alpha^{s-1}\gamma\ d\alpha\right|
\\
\le
\frac12\|\gamma\|_{s-2}\left\|\left[H,\left(\frac{2\pi V_W}{L}-\frac{4A\pi^2}{L^2}\gamma\right)\right]\partial_\alpha^{s-1}\gamma\right\|_1
\\
\le C\|\gamma\|_{s-2}\left\|\frac{2\pi V_W}{L}-\frac{4A\pi^2}{L^2}\gamma\right\|_2\|\gamma\|_{s-2}\le  C_1\exp(C_2 E).
\end{multline}

Next, we rewrite the sixth term on the right-hand side of \eqref{de2}:
\begin{equation}\nonumber
-\frac{2\rho\widetilde{A}\pi}{L}\int_0^{2\pi} H\partial_\alpha^{s-1} \gamma
H\partial_\alpha^{s-1}\gamma_{ t}d\alpha
=-\frac{\rho\widetilde{A}\pi}{L}\frac{d}{dt}\int_0^{2\pi}\left( H\partial_\alpha^{s-1} \gamma\right)^2
\ d\alpha.
\end{equation}

We integrate by parts in the seventh term on the right-hand side of \eqref{de2}, and combine this with the eighth term on
the right-hand side, finding
\begin{multline}\nonumber
\Bigg|-\frac{4\rho\pi^{2}\widetilde{A}}{L^2}
\int_0^{2\pi}\left( H\partial_\alpha^{s-1} \gamma\right)
V_W  H\partial_\alpha^{s} \gamma \ d\alpha
\\
-\frac{4(s-2)\rho\pi^{2}\widetilde{A}}{L^2}
\int_0^{2\pi}\left( H\partial_\alpha^{s-1} \gamma\right)
\partial_\alpha V_W ( H\partial_\alpha^{s-1} \gamma)\ d\alpha\ \Bigg|
\\
=\left|\frac{(-4s+10)\rho\pi^{2}\widetilde{A}}{L^2}
\int_0^{2\pi}\left( H\partial_\alpha^{s-1} \gamma\right)^2\partial_\alpha V_W \ d\alpha\right|
\le  C_1\exp(C_2 E).
\end{multline}

Combining these calculations, we are able to conclude that
\begin{multline}\label{de22}
\frac{dE_{2}}{dt}=-
\sigma\bar{A} \int_0^{2\pi}\left( H\partial_\alpha^{s+1} \gamma\right) \partial_\alpha^{s}\theta  \ d\alpha
-\int_0^{2\pi} \left(H\partial_\alpha^{s-1} \gamma\right)\frac{3\sigma\bar{A}\theta_\alpha^2}{2}\partial_\alpha^s\theta \ d\alpha
\\
+\int_0^{2\pi} \frac{4\pi\rho\widetilde{A}V_W^2}{L}\left(H\partial_\alpha^{s-1} \gamma\right)\partial_\alpha^s\theta \ d\alpha
-\int_0^{2\pi} \left(H\partial_\alpha^{s} \gamma\right)\lambda\partial_\alpha^{s-1}\theta \ d\alpha
\\
-\frac{\rho\widetilde{A}\pi}{L}\frac{d}{dt}\int_0^{2\pi}\left( H\partial_\alpha^{s-1} \gamma\right)^2
\ d\alpha+\Psi_2+\Psi_3.
\end{multline}
The collection of terms $\Psi_{2}$ is defined as
\begin{multline*}
\Psi_2=\int_0^{2\pi} \left(H\partial_\alpha^{s-1} \gamma\right)\left(\frac{2\pi V_W}{L}-\frac{4A\pi^2}{L^2}\gamma\right)\partial_\alpha^{s-1}\gamma\ d\alpha
\\
+\int_0^{2\pi}(\Lambda^{1/2}\partial_\alpha^{s-2}\gamma)\Lambda^{1/2}\Phi_2\ d\alpha+\int_0^{2\pi}(H\partial_\alpha^{s-1}\gamma)\tilde{R}_1\ d\alpha
+\rho\int_0^{2\pi}(H\partial_\alpha^{s-1}\gamma)\tilde{\Phi}_2\ d\alpha,
\end{multline*}
and has the estimate $|\Psi_2|\le C_1\exp(C_2 E).$
The collection of terms $\Psi_{3}$ is proportional to $\rho,$ and is defined as
\begin{multline}\nonumber
\Psi_3=-\frac{4\rho\pi^{2}\widetilde{A}}{L^2}
\int_0^{2\pi}\left( H\partial_\alpha^{s-1} \gamma\right)
V_W  (H\partial_\alpha^{s} \gamma) \ d\alpha
\\
-\frac{4(s-2)\rho\pi^{2}\widetilde{A}}{L^2}
\int_0^{2\pi}\left( H\partial_\alpha^{s-1} \gamma\right)^{2}
(\partial_\alpha V_W)\ d\alpha,
\end{multline}
and has the estimate $|\Psi_3|\le  C_1\exp(C_2 E).$

Next, we compute $\frac{d(\rho E_4)}{dt}:$
\begin{align}\label{de3}
\frac{d(\rho E_4)}{dt}=-\frac{d}{dt}\left(\frac{\rho\pi \widetilde{A}}{L}\right)\int_0^{2\pi}(H\partial_\alpha^{s-1} \gamma)^2 \ d\alpha
+\frac{\rho\pi \widetilde{A}}{L}\frac{d}{dt}\int_0^{2\pi}(H\partial_\alpha^{s-1} \gamma)^2 \ d\alpha.
\end{align}
When adding, \eqref{de3} will provide a cancellation with \eqref{de22}.
Now we  add \eqref{de13}, \eqref{dE1prime}, \eqref{de22}, and \eqref{de3},  concluding that
\begin{multline}\label{dE123}
\frac{d(\sigma E_1)}{dt}+\frac{dE_{2}}{dt}+\frac{dE_{3}}{dt}+\frac{d(\rho E_{4})}{dt}\\
=-\int_0^{2\pi} \left(H\partial_\alpha^{s-1} \gamma\right)\frac{\sigma\bar{A}(\theta_\alpha^2+1)}{2}\partial_\alpha^s\theta \ d\alpha
+\int_0^{2\pi} \frac{4\pi\rho\widetilde{A}V_W^2}{L}\left(H\partial_\alpha^{s-1} \gamma\right)\partial_\alpha^s\theta \ d\alpha\\
-\frac{d}{dt}\left(\frac{\rho\pi \widetilde{A}}{L}\right)
\int_0^{2\pi}(H\partial_\alpha^{s-1} \gamma)^2 \ d\alpha+\Psi_1+\Psi_2+\Psi_3
+C_{1}\exp(C_{2}E).
\end{multline}
We note that the terms which contain $L_{t}$  are immediately
bounded, using the definition of the energy and Lemma \ref{priori}, thus
\[
\left|-\frac{d}{dt}\left(\frac{\rho\pi \widetilde{A}}{L}\right)\int_0^{2\pi}(H\partial_\alpha^{s-1} \gamma)^2 \ d\alpha+\Psi_1+\Psi_2+\Psi_3\right|
\le C_1\exp(C_2E),
\]
where $C_1,C_2$ are independent of $\rho_0$ and $\sigma.$

To cancel the first  two terms of the right-hand side of \eqref{dE123}, we will use $E_{5},$ $E_{6}$ and $E_{7},$ defined as
\begin{align}\nonumber
E_{5}=\frac12\int_0^{2\pi}\sqrt{(\theta_\alpha^2+1)/2}(\partial_\alpha^{s-3}\gamma)\Lambda\left(\sqrt{(\theta_\alpha^2+1)/2}(\partial_\alpha^{s-3}\gamma)\right)\ d\alpha,
\end{align}
\begin{align}\nonumber
E_{6}=\int_0^{2\pi}\left(\frac{(\theta_\alpha^2+1)\widetilde{A}\pi}{2L}\right)(H\partial_\alpha^{s-2}\gamma)^2\ d\alpha,
\end{align}
\begin{align}\nonumber
E_{7}=\frac{ L\tilde{A}}{\pi}\int_0^{2\pi}V_W^2(\partial_\alpha^{s-1}\theta)^2\ d\alpha.
\end{align}

We proceed by computing $dE_{5}/dt$:
\begin{multline}\nonumber
\frac{dE_{5}}{dt}=\int_0^{2\pi}\frac{\partial_t\sqrt{(\theta_\alpha^2+1)/2}}{2\sqrt{(\theta_\alpha^2+1)/2}}(\partial_\alpha^{s-3}\gamma)\Lambda\left(\sqrt{(\theta_\alpha^2+1)/2}(\partial_\alpha^{s-3}\gamma)\right)\ d\alpha
\\
+\int_0^{2\pi}\sqrt{(\theta_\alpha^2+1)/2}(\partial_\alpha^{s-3}\gamma_t)\Lambda\left(\sqrt{(\theta_\alpha^2+1)/2}(\partial_\alpha^{s-3}\gamma)\right)\ d\alpha\\
\le\int_0^{2\pi}\sqrt{(\theta_\alpha^2+1)/2}(\partial_\alpha^{s-3}\gamma_t)\Lambda\left(\sqrt{(\theta_\alpha^2+1)/2}(\partial_\alpha^{s-3}\gamma)\right)\ d\alpha+C_1\exp(C_2E).
\end{multline}
To treat this, we first expand $\partial_{\alpha}^{s-3}\gamma_{t}$ in the second term on the right-hand side.  We write this as
\begin{equation}\nonumber
\partial_\alpha^{s-3}\gamma_t=-\sigma\bar{A} \partial_\alpha^{s+1}\theta -\frac{2\rho\widetilde{A}\pi}{L}H(\partial_\alpha^{s-2}\gamma_t)
-\frac{4\rho\widetilde{A}V_W\pi^2}{L^2} H(\partial_\alpha^{s-1}\gamma)+\Phi_3
\end{equation}
where
\begin{multline}\nonumber
\Phi_3=\partial_\alpha^{s-3}\left(-2AJ[z_d]\gamma_t+2\rho\left(S(\gamma_t)-\frac{2\pi\widetilde{A}\theta_\alpha}{L}J[z_d]\gamma_t\right)
+\rho\widetilde{R}_1+\rho\widetilde{R}_2+R\right)\\
+\partial_\alpha^{s-3}\left(\left(-\frac{3\sigma\bar{A}\theta_\alpha^2}{2}+\lambda+\frac{4\pi \rho\widetilde{A}V_W^2}{L}\right)\partial_\alpha^2\theta +\left(\frac{2\pi V_W}{L}-\frac{4A\pi^2}{L^2}\gamma\right)\partial_\alpha\gamma\right)
\\
-\sum_{j=0}^{s-4}\binom{s}{j}\frac{4\rho\widetilde{A}\partial^{s-2-j}_\alpha V_W\pi^2}{L^2} H(\partial_\alpha^{j+2}\gamma).
\end{multline}
We notice that $\gamma_t$ is well-defined by Assumption \ref{solvabilityAssumption}, and the terms here involving $\gamma_{t}$ are
bounded in $H^0$ when $s\ge 3$ since the operators $S$ and $J[z_d]$ are smoothing.
The collection $\Phi_3$ involves at most $s-1$ derivatives of $\theta$ and $s-2$ derivatives of $\gamma$, so $\|\Phi_3\|_{0}\le C_1\exp(C_2 E)$.
Next, we expand $\Lambda\left(\sqrt{(\theta_\alpha^2+1)/2}(\partial_\alpha^{s-3}\gamma)\right),$ writing it as
\begin{multline}\nonumber
\Lambda\left(\sqrt{(\theta_\alpha^2+1)/2}\partial_\alpha^{s-3}\gamma\right)\\
=H\left((\sqrt{(\theta_\alpha^2+1)/2})_\alpha\partial_\alpha^{s-3}\gamma\right)
+H\left(\sqrt{(\theta_\alpha^2+1)/2}\partial_\alpha^{s-2}\gamma\right)
\\
=H\left((\sqrt{(\theta_\alpha^2+1)/2})_\alpha\partial_\alpha^{s-3}\gamma\right)
+\left[H,\sqrt{(\theta_\alpha^2+1)/2}\right]\partial_\alpha^{s-2}\gamma
\\
+\sqrt{(\theta_\alpha^2+1)/2}H\partial_\alpha^{s-2}\gamma.
\end{multline}
Thus
\begin{multline}\label{dE42}
\frac{dE_5}{dt}=\int_0^{2\pi}\frac{\sigma\bar{A}(\theta_\alpha^2+1)}{2}\partial_\alpha^{s}\theta H\partial_\alpha^{s-1}\gamma d\alpha
\\
-\int_0^{2\pi}\sqrt{(\theta_\alpha^2+1)/2}\left(\frac{2\rho\widetilde{A}\pi}{L}H\partial_\alpha^{s-2}\gamma_t\right)
H\left((\sqrt{(\theta_\alpha^2+1)/2})_\alpha\partial_\alpha^{s-3}\gamma\right)d\alpha\\
-\int_0^{2\pi}\frac{(\theta_\alpha^2+1)\rho\widetilde{A}\pi}{L}H(\partial_\alpha^{s-2}\gamma_t) H\partial_\alpha^{s-2}\gamma d\alpha+\Psi_{5},
\end{multline}
and this will provide a cancellation with \eqref{dE123}.
Here, $\Psi_{5}$ is the collection of terms
\begin{multline*}
\Psi_5=\int_0^{2\pi}\frac{\sigma\bar{A}(\theta_\alpha^2+1)_\alpha}{2}\partial_\alpha^{s}\theta H\partial_\alpha^{s-2}\gamma d\alpha
\\
+\int_0^{2\pi}
\sigma\bar{A}\partial_\alpha^s\theta\partial_\alpha\left(\sqrt{(\theta_\alpha^2+1)/2}H\left((\sqrt{(\theta_\alpha^2+1)/2})_\alpha\partial_\alpha^{s-3}\gamma\right)\right)
d\alpha\\
+\int_0^{2\pi}
\sigma\bar{A}\partial_\alpha^s\theta\partial_\alpha\left(\sqrt{(\theta_\alpha^2+1)/2}\left[H,\sqrt{(\theta_\alpha^2+1)/2}\right]\partial_\alpha^{s-2}\gamma\right)d\alpha\\
+\int_0^{2\pi}\sqrt{(\theta_\alpha^2+1)/2}(-\frac{4\rho\widetilde{A}V_W\pi^2}{L^2} H(\partial_\alpha^{s-1}\gamma)+\Phi_3)\Lambda\left(\sqrt{(\theta_\alpha^2+1)/2}\partial_\alpha^{s-3}\gamma\right)d\alpha\\
-\int_0^{2\pi}\sqrt{(\theta_\alpha^2+1)/2}(\frac{2\rho\widetilde{A}\pi}{L}H\partial_\alpha^{s-2}\gamma_t)\left[H,\sqrt{(\theta_\alpha^2+1)/2}\right]\partial_\alpha^{s-2}\gamma d\alpha.
\end{multline*}
We notice that $\|\gamma_t\|_{s-4}\le C_1\exp(C_2 E)$.
The collection $\Psi_5$ involves at most $s$ derivatives of $\sigma\theta$, $s-1$ derivatives of $\theta$ and $s-3/2$ derivatives of $\gamma$, so $|\Psi_5|\le C_1\exp(C_2 E)$.

We next integrate by parts in the second term of the right-hand side of \eqref{dE42}, and again using the fact again that $\|\gamma_t\|_{s-4}\le C_1\exp(C_2 E),$
we have
\begin{multline}\nonumber
\left|-\int_0^{2\pi}H(\partial_\alpha^{s-4}\gamma_t)\partial_\alpha^2\left(\frac{2\rho\widetilde{A}\pi}{L}\sqrt{(\theta_\alpha^2+1)/2}H\left((\sqrt{(\theta_\alpha^2+1)/2})_\alpha\partial_\alpha^{s-3}\gamma\right)\right)d\alpha\right|
\\
\le C_1\exp(C_2 E).
\end{multline}
We then integrate by parts in the third term of the right-hand side of \eqref{dE42}, finding
\begin{multline}\label{E4remaider}
-\int_0^{2\pi}\frac{(\theta_\alpha^2+1)\rho\widetilde{A}\pi}{L}H(\partial_\alpha^{s-2}\gamma_t)(H\partial_\alpha^{s-2}\gamma)d\alpha\\
=-\frac{d}{dt}\int_0^{2\pi}\frac{(\theta_\alpha^2+1)\rho\widetilde{A}\pi}{2L}(H\partial_\alpha^{s-2}\gamma)^2d\alpha+\int_0^{2\pi}\frac{d}{dt}\left(\frac{(\theta_\alpha^2+1)\rho\widetilde{A}\pi}{2L}\right)(H\partial_\alpha^{s-2}\gamma)^2d\alpha\\
=-\frac{d(\rho E_6)}{dt}+\int_0^{2\pi}\frac{d}{dt}\left(\frac{(\theta_\alpha^2+1)\rho\widetilde{A}\pi}{2L}\right)(H\partial_\alpha^{s-2}\gamma)^2d\alpha.
\end{multline}
The second term of the right-hand side of \eqref{E4remaider} involves $L_{t},$ at most one derivative of $\theta_{t},$ and at most $s-2$ derivatives of $\gamma.$
To estimate this, since $\|\theta_t\|_{s-5/2}\le C_1\exp(C_2 E)$ and $L_t$ is bounded,  for $s>7/2$ we have
\[
\left|\int_0^{2\pi}\frac{d}{dt}\left(\frac{(\theta_\alpha^2+1)\rho\widetilde{A}\pi}{2L}\right)(H\partial_\alpha^{s-2}\gamma)^2d\alpha\right|\le C_1\exp(C_2 E).
\]
We then compute $\frac{dE_7}{dt}$ (which is similar to ${dE_3}/{dt}$), finding
\begin{multline}\label{de6}
\frac{d(\rho E_7)}{dt}\\=\int_0^{2\pi} \left(\frac{L\rho\tilde{A}V_W^2}{\pi}\right)_t\left(\partial_\alpha^{s-1} \theta\right)\partial_\alpha^{s-1}\theta d\alpha+\int_0^{2\pi} \frac{2L\rho\tilde{A}V_W^2}{\pi}\left(\partial_\alpha^{s-1} \theta\right)\partial_\alpha^{s-1}\theta_t d\alpha \\
\le \int_0^{2\pi} \left(\frac{L\rho\tilde{A}V_W^2}{\pi}\right)_t\left(\partial_\alpha^{s-1} \theta\right)\partial_\alpha^{s-1}\theta d\alpha+\int_0^{2\pi} \frac{4\pi\rho\tilde{A}V_W^2}{L}\left(\partial_\alpha^{s-1} \theta\right)(H\partial_\alpha^{s}\gamma) d\alpha
\\+C_1\exp(C_2 E)\\
\le\int_0^{2\pi} \left(\frac{L\rho\tilde{A}V_W^2}{\pi}\right)_t\left(\partial_\alpha^{s-1} \theta\right)\partial_\alpha^{s-1}\theta d\alpha-\int_0^{2\pi} \frac{4\pi\rho\tilde{A}V_W^2}{L}\left(\partial_\alpha^{s-1} \theta\right)(H\partial_\alpha^{s}\gamma) d\alpha
\\+C_1\exp(C_2 E).
\end{multline}
 To estimate the first term on the right-hand side of \eqref{de6}, we focus on $(V_W)_t$. Since $s$ is large enough, we have uniform bounds on
$\gamma_{t},$  $\theta_t$ and $L_t$, and thus $(V_W)_t$ is also uniformly bounded. Therefore
\[
\left|\int_0^{2\pi} \left(\frac{L\rho\tilde{A}V_W^2}{\pi}\right)_t\left(\partial_\alpha^{s-1} \theta\right)\partial_\alpha^{s-1}\theta d\alpha\right|\le  C_1\exp(C_2 E).
\]
We integrate by parts in the second term of the right-hand side of \eqref{de6},
\begin{align*}
&\int_0^{2\pi} \frac{4\pi\rho\tilde{A}V_W^2}{L}\partial_\alpha^{s-1} \theta H\partial_\alpha^{s}\gamma d\alpha\\&=-\int_0^{2\pi} \frac{4\pi\rho\tilde{A}V_W^2}{L}\partial_\alpha^{s} \theta H\partial_\alpha^{s-1}\gamma d\alpha-\int_0^{2\pi} \frac{4\pi\rho\tilde{A}(V_W^2)_\alpha}{L} \partial_\alpha^{s-1} \theta H\partial_\alpha^{s-1}\gamma d\alpha\\
&\le -\int_0^{2\pi} \frac{4\pi\rho\tilde{A}V_W^2}{L} \partial_\alpha^{s} \theta H\partial_\alpha^{s-1}\gamma d\alpha + C_1\exp(C_2 E).
\end{align*}

We now are in a position to combine the time derivatives of $E_0,$ $\sigma E_1,$ $E_2,$ $E_3$ and $\rho E_4,$ $ E_5,$ $\rho E_6,$ and $\rho E_{7}.$
 We define the total energy
 \begin{align}
 E_{total}&=E_0+\sigma E_1+E_2+E_3+E_5+\rho E_4+\rho E_6+\rho E_7\nonumber\\
 &=E_0+\sigma E_1+E_2+E_3+E_5+2\pi\rho_0 E_4/L+2\pi\rho_0 E_6/L+2\pi\rho_0 E_7/L,
 \nonumber
 \end{align}
  and we have concluded the following estimate:
\begin{equation}\nonumber
\frac{dE_{total}}{dt}\le C_1\exp(C_2E_{total}).
\end{equation}
Here $C_1,C_2$ are independent of $\rho$ and $\sigma.$
Our energy bound implies
\begin{equation}\nonumber
E_{total}(t)\le \frac{-\ln(e^{-C_2E_{total}(0)}-C_1C_2t)}{C_2}.
\end{equation}
This completes the proof of the theorem.
\end{proof}

\section{Cauchy sequence as $(\sigma,\rho_0)\rightarrow (0^+,0^+)$}\label{limitSection}
Under Assumption \ref{solvabilityAssumption}, for any $(\sigma,\rho_0)$ we have proved above that solutions
$(\theta^{\sigma,\rho_{0}},\gamma^{\sigma,\rho_{0}})$ are in $C([0,T], \mathcal{O}),$ with the time, $T,$ independent of the parameters.
Over this time interval, the norm
\begin{equation}\nonumber
\sigma\|\theta^{\sigma,\rho_0}\|_s^2+\|\theta^{\sigma,\rho_0}\|_{s-1}^2+\|\gamma^{\sigma,\rho_0}\|_{s-3/2}^2+\rho_0\|\gamma^{\sigma,\rho_0}\|_{s-1}^2
\end{equation}
is uniformly bounded with respect to the parameters.
This implies that for $\rho_0=0,$ $\sigma=0$, the norm
$\|\theta\|_{s-1}^2+\|\gamma\|_{s-3/2}^2$ is bounded; without the elastic parameters, this is the case of the vortex sheet with surface tension.
We now demonstrate that this vortex sheet with surface tension
 is the limit of hydroelastic waves as
the density of elastic sheet and bending energy goes to zero by demonstrating that the  solutions of $(\theta^{\sigma,\rho_0},\gamma^{\sigma,\rho_0}) $ are
a Cauchy sequence as the parameters vanish.  In what follows, we continue to take $s$ sufficiently large.
\begin{theorem}\label{Cauchytheorem}
For any $\eta>0$, there exists  $\delta>0$ such that if
\begin{equation}\nonumber
\sqrt{|\rho_0-\rho_0^\prime|^2+|\sigma-\sigma^\prime|^2}<\delta,
\end{equation}
then
\begin{equation}\nonumber
\|\theta-\theta^\prime\|_2+\|\gamma-\gamma^\prime \|_{3/2}<\eta,
\end{equation}
where $(\theta,\gamma) $ and  $(\theta^\prime,\gamma^\prime) $ refer to $(\theta^{\sigma,\rho_0},\gamma^{\sigma,\rho_0}) $ and $(\theta^{\sigma^\prime,\rho_0^\prime},\gamma^{\sigma^\prime,\rho_0^\prime}) $, respectively, with the superscripts omitted.
\end{theorem}
Before proving the theorem, we will establish some important Lipschitz estimates.
For two given solutions $(\theta,\gamma)$ and  $(\theta^\prime,\gamma^\prime)$ in $C([0,T], \mathcal{O})$, we define differences
$\delta \theta=\theta-\theta^\prime$ and $\delta\gamma=\gamma-\gamma^\prime.$  The evolution equation for  $\delta \theta $ is
\begin{equation}\label{evolution of deltatheta}
\delta\theta_t=\frac{2\pi^2}{L^2}H(\delta\gamma_\alpha)+\frac{2\pi}{L}V_W^\prime\delta\theta_\alpha+\frac{2\pi}{L}\theta_\alpha\delta V_W+B_1.
\end{equation}
where $B_1$ is defined by
\begin{align}\nonumber
B_1=\left(\frac{2\pi}{L}\bm{m}\cdot \bm{n}-\frac{2\pi}{L^\prime}\bm{m}^\prime\cdot \bm{n}^\prime\right)+\left(\frac{2\pi^2}{L^2}-\frac{2\pi^2}{(L^\prime)^2}\right)H(\gamma^\prime_\alpha)+\left(\frac{2\pi}{L}-\frac{2\pi }{L^\prime}\right)V_W^\prime\theta^\prime_\alpha.
\end{align}
It is understood in each case that $L$ and $L^\prime$ indicate the lengths of the curves associated to $\theta$ and $\theta\prime,$ respectively, and
$\delta L = L-L^\prime;$ other quantities such as $\delta V_{W}$ are defined in the same manner.
Similarly, the evolution of  $\delta \gamma $ is
\begin{multline}\label{deltagammaend}
\delta\gamma_t=\lambda\delta\theta_{\alpha\alpha}-\sigma\bar{A} \partial_\alpha^4\delta\theta -\frac{3\sigma\bar{A}\theta_\alpha^2}{2}\delta\theta_{\alpha\alpha} +\left(\frac{2\pi V_W}{L}-\frac{4A\pi^2}{L^2}\gamma\right)\delta\gamma_\alpha+\frac{4\pi\rho\widetilde{A}}{L} V_W ^2\delta\theta_{\alpha\alpha}\\-\frac{2\rho\widetilde{A}\pi}{L}H(\delta\gamma_{\alpha t})
-\frac{4\rho\widetilde{A}V_W\pi^2}{L^2} H(\delta\gamma_{\alpha\alpha})+T[\theta]\delta\gamma_t+\rho(\widetilde{R}_1
-\widetilde{R}_1^\prime)+B_2+\rho\tilde{B}_{2},
\end{multline}
where $B_2$ is defined by
\begin{align*}
 B_2&=\left(-\frac{3\bar{A}(\sigma\theta_\alpha^2-\sigma^\prime(\theta_\alpha^\prime)^2)}{2}+(\lambda-\lambda^\prime)+\frac{4\pi\rho\widetilde{A}}{L} V_W ^2-\frac{4\pi\rho\widetilde{A}}{L^\prime} (V^\prime_W) ^2\right)\theta^\prime_{\alpha\alpha}\\
 &+\left(\frac{2\pi V_W}{L}-\frac{2\pi V_W^\prime}{L^\prime}-\frac{4A\pi^2}{L^2}\gamma+\frac{4A\pi^2}{(L^\prime)^2}\gamma^\prime\right)\gamma_\alpha^\prime+2A(J[z_d^\prime]-J[z_d])\gamma^\prime_t\\
 &+\left(\frac{2\rho^\prime\widetilde{A}\pi}{L^\prime}-\frac{2\rho\widetilde{A}\pi}{L}\right)H(\gamma^\prime_{\alpha t})+\left(\frac{4\rho^\prime\widetilde{A}V_W^\prime\pi^2}{(L^\prime)^2}-\frac{4\rho\widetilde{A}V_W\pi^2}{L^2}\right) H(\gamma^\prime_{\alpha\alpha})
\\&+2(\rho S[z_d]-\rho^\prime S[z_d^\prime])\gamma^\prime_t
-\left(\frac{4\pi\rho\widetilde{A}}{L}-\frac{4\pi\rho^\prime\widetilde{A}}{L^\prime}\right)\theta^\prime_\alpha J[z_d^\prime]\gamma^\prime_t\\&
-(\sigma-\sigma^\prime)\bar{A}\partial_\alpha^4\theta^\prime+(\rho-\rho^\prime)\widetilde{R}_1^\prime+(\rho-\rho^\prime)\widetilde{R}_2^\prime+R-R^\prime,
\end{align*}
and $\tilde{B}_{2}$ is defined by
\begin{align}\nonumber
\tilde{B}_2=\widetilde{R}_2-\widetilde{R}_2^\prime-\left(\frac{4\pi\widetilde{A}\delta\theta_\alpha}{L}J[z_d^\prime]\right)\gamma^\prime_{t}.
\end{align}

We begin to define the energy functional for the difference of two solutions as
\begin{equation}\nonumber
E_d(t)=Z_0(t)+\sigma Z_1(t)+Z_2+Z_3(t)+\rho Z_4(t),
\end{equation}
where
\begin{align}
\nonumber
&Z_0(t)=\frac{1}{2}(\delta L)^2+\frac12\int_0^{2\pi}(\delta\theta)^2+(\delta\gamma)^2 d\alpha,\\
\nonumber
&Z_1(t)=\frac{L^2 \bar{A}}{4\pi^2 }\int_0^{2\pi}(\partial_\alpha^3\delta\theta)^2 d\alpha,\\
\nonumber
&Z_2(t)=\frac12\int_0^{2\pi}(\partial_\alpha \delta\gamma)( H\partial_\alpha^{2} \delta\gamma) d\alpha,\\
\nonumber
&Z_3=\left(\frac{\tau L}{\pi(\rho_1+\rho_2)}+\frac{
\sigma \bar{A}L^2}{8\pi^2}\right)\int_0^{2\pi}(\partial_\alpha^{2}\delta\theta)^2d\alpha,\\
\nonumber
&Z_4(t)=\frac{\pi \widetilde{A}}{L}\int_0^{2\pi}(H\partial_\alpha^{2} \delta\gamma)^2 d\alpha.
\end{align}
We will estimate this energy, $E_{d},$ in due course.  But first we will continue with some auxiliary estimates.

We next establish estimates for $B_{1},$ $B_{2},$ $\widetilde{B}_{2}$ and similar quantities, in the following Lemma.
\begin{lemma}\label{Lippriori}
Assuming Assumption \ref{solvabilityAssumption}, and for given $(\sigma,\rho_0),(\sigma^\prime,\rho_0^\prime),$
the quantities associated to the solutions $(\theta,\gamma) $,  $(\theta^\prime,\gamma^\prime) $ in $C([0,T], \mathcal{O})$ satisfy the following estimates:
\begin{align}
\nonumber&|\delta L_t|\le CE_d^{1/2}\\
\nonumber&\|\delta \theta_t\|_0\le CE_d^{1/2}\\
\nonumber
&\|B_1\|_i\le C(\|\delta\theta\|_i+C\|\delta\gamma\|_{3/2}+|\delta L|),\quad i=2,3,\\
\nonumber
&\|B_2\|_{3/2}\le  CE_d^{1/2}+C(|\rho_0-\rho_0^\prime|+|\sigma-\sigma^\prime|),\\
\nonumber
&\|\tilde{B}_2\|_{1}\le  CE_d^{1/2},\\
&\|\widetilde{R}_1-\widetilde{R}_1^\prime\|_1\le  C(\|\delta\theta\|_2+C\|\delta\gamma\|_{2}+|\delta L|).\label{R1Lip}
\end{align}
\end{lemma}
\begin{proof}
It is important to be able to get the Lipschitz estimates on the various quantities corresponding to Lemma \ref{priori} associated to $(\theta-\theta^\prime, \gamma-\gamma^\prime)$ with $i=2,3$. The simplest quantities are the unit tangent and normal vectors which are bounded by standard Lipschitz estimates for sine and cosine:
\begin{align}\nonumber
\|\bm{t}-\bm{t}^\prime\|_i=\|(\cos (\theta)-\cos(\theta^\prime), \sin(\theta)-\sin(\theta^\prime))\|_i\le C\|\theta-\theta^\prime\|_i,
\end{align}
 and similarly,
 \[
 \|\bm{n}-\bm{n}^\prime\|_i\le C\|\theta-\theta^\prime\|_i.
 \]
Since $z_a=\frac{L}{2\pi}\mathcal{C}(\bm{t})$, a bound for $z_\alpha-z_\alpha^\prime$ follows:
\begin{align}\nonumber
\|z_\alpha-z_\alpha^\prime\|_i\le
\left\|\frac{L-L^\prime}{2\pi}\mathcal{C}(\bm{t})\right\|_i
+\frac{L^\prime}{2\pi}\|\mathcal{C}(\bm{t})-\mathcal{C}(\bm{t}^\prime)\|_3\le C(\|\delta\theta\|_i+C\|\delta\gamma\|_{3/2}+|\delta L \|).
\end{align}
Next we estimate $\bm{m}-\bm{m}^\prime$. We rewrite
\[\mathcal{C}(\bm{m}-\bm{m}^\prime)= \uppercase\expandafter{\romannumeral1}
+ \uppercase\expandafter{\romannumeral2},
\]
where
\begin{align*}
 \uppercase\expandafter{\romannumeral1}=z_\alpha K[z_d]\left(\left(\frac{\gamma}{z_\alpha}\right)_\alpha\right)-z_\alpha^\prime K[z_d^\prime]\left(\left(\frac{\gamma^\prime}{z_\alpha^\prime}\right)_\alpha\right),\\
 \uppercase\expandafter{\romannumeral2}=\frac{z_\alpha}{2i} \left[H,\frac{1}{z_\alpha^2}\right]\left(z_\alpha\left(\frac{\gamma}{z_\alpha}\right)_\alpha\right)
 -\frac{z_\alpha^\prime}{2i} \left[H,\frac{1}{(z_\alpha^\prime)^2}\right]\left(z_\alpha^\prime\left(\frac{\gamma^\prime}{z_\alpha^\prime}\right)_\alpha\right).
\end{align*}
We rewrite $ \uppercase\expandafter{\romannumeral1}$ by adding and subtracting:
\begin{multline}\nonumber
 \uppercase\expandafter{\romannumeral1}=(z_\alpha-z_\alpha^\prime) K[z_d]\left(\left(\frac{\gamma}{z_\alpha}\right)_\alpha\right)
 +z_\alpha^\prime (K[z_d]-K[z_d^\prime])\left(\left(\frac{\gamma}{z_\alpha}\right)_\alpha\right)
\\
 +z_\alpha^\prime K[z_d^\prime]\left(\left(\frac{\gamma}{z_\alpha}\right)_\alpha
 -\left(\frac{\gamma^\prime}{z_\alpha^\prime}\right)_\alpha\right).
\end{multline}
 With our uniform bounds on $\theta$ and $\gamma$, and using Lemma \ref{lemma2} and Lemma \ref{lemma5}, the
  first term on the right-hand side  is bounded in $H^i$ by $C(\|\delta\theta\|_i+C\|\delta\gamma\|_{3/2}+|\delta L|).$ For the second term, we apply Lemma \ref{lemma6}, and we thus see that this term is also bounded in $H^{i}$
  by $C(\|\delta\theta\|_i+C\|\delta\gamma\|_{3/2}+|\delta L|)$. For the third term on the right-hand side, its norm in $H^i$ is bounded by $C(\|\delta\theta\|_i+C\|\delta\gamma\|_{3/2}+|\delta L|)$ since $K[z_d^\prime]$ is a smoothing operator by Lemma \ref{lemma5}.

 We now consider the term  $\uppercase\expandafter{\romannumeral2}$. We start by adding and subtracting:
 \begin{multline*}
 \uppercase\expandafter{\romannumeral2}=\frac{z_\alpha-z_\alpha^\prime}{2i}
 \left[H,\frac{1}{z_\alpha^2}\right]\left(z_\alpha\left(\frac{\gamma}{z_\alpha}\right)_\alpha\right)
 +\frac{z_\alpha^\prime}{2i}
 \left[H,\frac{1}{z_\alpha^2}-\frac{1}{(z_\alpha^\prime)^2}\right]\left(z_\alpha\left(\frac{\gamma}{z_\alpha}\right)_\alpha\right)
 \\
 +\frac{z_\alpha^\prime}{2i} \left[H,\frac{1}{(z_\alpha^\prime)^2}\right]
 \left(z_\alpha\left(\frac{\gamma}{z_\alpha}\right)_\alpha-z_\alpha^\prime\left(\frac{\gamma^\prime}{z_\alpha^\prime}\right)_\alpha\right).
 \end{multline*}
 The first term on the right-hand side can be immediately bounded in $H^{i}$ by $C(\|\delta\theta\|_i+C\|\delta\gamma\|_{3/2}+|\delta L|)$.
 The second and third terms can be bounded in $H^{i}$ by  $C(\|\delta\theta\|_i+C\|\delta\gamma\|_{3/2}+|\delta L|)$ by Lemma \ref{lemma7}.
 Each of the terms on the right-hand side can clearly be bounded in $H^{i}$ by $C(\|\delta\theta\|_i+C\|\delta\gamma\|_{3/2}+|\delta L|)$. This completes the estimate of $\bm{m}-\bm{m}^\prime$. We conclude
 \begin{equation}\nonumber
 \|\bm{m}-\bm{m}^\prime\|_i\le C(\|\delta\theta\|_i+C\|\delta\gamma\|_{3/2}+|\delta L|).
 \end{equation}
Now we can conclude that
\begin{equation}\nonumber
\|B_1\|_i\le C(\|\delta\theta\|_i+C\|\delta\gamma\|_{3/2}+|\delta L|).
\end{equation}

After adding and subtracting several times, and by the above estimates on $\bm{t}-\bm{t}^\prime$ and $\bm{m}-\bm{m}^\prime$, and by equation \eqref{V_Wep}, the following estimate can be found:
\[
\|V_W-V_W^\prime\|_{2}\le C E_d^{1/2}.
\]
We also may conclude that
\[
\|\delta\theta_t\|_{0}\le C E_d^{1/2}.
\]
 Now we estimate $\|\bm{W}-\bm{W}^\prime\|_{3/2}$.  We first expand this, as we have several times before:
 \begin{multline}\nonumber
 \mathcal{C}^*(\bm{W})-\mathcal{C}^*(\bm{W}^\prime)
 =(K[z_d]-K[z_d^\prime])\gamma\\+\frac{1}{2i}H(\gamma/z_\alpha-\gamma^\prime/z_\alpha)
 +K[z_d^\prime](\gamma-\gamma^\prime)+\frac{1}{2i}H(\gamma^\prime/z_\alpha-\gamma^\prime/z_\alpha^\prime).
 \end{multline}
By Lemma \ref{lemma5} and Lemma \ref{lemma6}, $\|\bm{W}-\bm{W}^\prime\|_{3/2}$ is thus bounded by $C E_d^{1/2}$.
This immediately yields the fact that $\|U-U^\prime\|_{3/2}$ is also
bounded by $C E_d^{1/2}$ since $U=\bm{W}\cdot\bm{n}$.  From formula \eqref{LtDef} and the Schwartz inequality, we immediately find the following estimate:
\begin{equation}\nonumber
|L_t-L_t^\prime|\le C E_d^{1/2}.{}
\end{equation}
From $z_t=\mathcal{C}(U\bm{n}+V\bm{t})$ , we have
  \begin{equation}\nonumber
\|z_t-z_t^\prime\|_{3/2}\le C E_d^{1/2}.
\end{equation}
As in the proof of Lemma \ref{priori}, we could continue to get the Lipschitz estimates on $R_1,$ $R_3,$ $R_5,$ $\widetilde{R}_2,$ and $R.$
There are many Hilbert commutators and terms involving
 $K[z_d]$ in the definition of $B_2$ and we only show the estimate for one such term  in $\|R_3-R_3^\prime\|_{3/2}$ as an example. In the term on which we focus, we use the smoothing properties of Hilbert commutators carefully; for other terms,
 the similar properties of $K[z_{d}]$ are used.  For the term we select, we add and subtract:
\begin{multline*}
\left[H,\frac{1}{z_\alpha}\right]\left(z_\alpha\left(\frac{\gamma z_{t\alpha}}{z_\alpha}\right)_\alpha\right)
-\left[H,\frac{1}{z_\alpha^\prime}\right]\left(z_\alpha^\prime\left(\frac{\gamma^\prime z_{t\alpha}^\prime}{z_\alpha^\prime}\right)_\alpha\right)
=\left[H,\frac{1}{z_\alpha}-\frac{1}{z_\alpha^\prime}\right]\left(z_\alpha\left(\frac{\gamma z_{t\alpha}}{z_\alpha}\right)_\alpha\right)\\
+\left[H,\frac{1}{z_\alpha^\prime}\right]\left((z_\alpha-z_\alpha^\prime)\left(\frac{\gamma z_{t\alpha}^\prime}{z_\alpha^\prime}\right)_\alpha\right)
+\left[H,\frac{1}{z_\alpha^\prime}\right]\left(z_\alpha^\prime\left(\frac{(\gamma-\gamma^\prime) z_{t\alpha}^\prime}{z_\alpha^\prime}\right)_\alpha\right)
\\+\left[H,\frac{1}{z_\alpha^\prime}\right]\left(z_\alpha^\prime\left(\gamma^\prime(\frac{ z_{t\alpha}}{z_\alpha}-\frac{ z_{t\alpha}^\prime}{z_\alpha^\prime})\right)_\alpha\right).
\end{multline*}
The first term  and the second term on the right-hand side are bounded in $H^{2}$ by $CE_d^{1/2}$ since $\|z_\alpha-z_\alpha^\prime\|_{3/2}$ is bounded by  $CE_d^{1/2}$ and $z_\alpha\left(\frac{\gamma z_{t\alpha}}{z_\alpha}\right)_\alpha$ is bounded in $H^{s-3}$.
The third term on the right-hand side is bounded in $H^{3/2}$ by $CE_d^{1/2}$ since by Lemma \ref{lemma7},
\[
\left\|\left[H,\frac{1}{z_\alpha^\prime}\right]\left(z_\alpha^\prime\left(\frac{(\gamma-\gamma^\prime) z_{t\alpha}^\prime}{z_\alpha^\prime}\right)_\alpha\right)\right\|_{3/2}\le C
\left\|\frac{1}{z_\alpha^\prime}\right\|_{3/2}\left\|\left(z_\alpha^\prime\left(\frac{(\gamma-\gamma^\prime) z_{t\alpha}^\prime}{z_\alpha^\prime}\right)_\alpha\right)\right\|_0.
\]
The fourth term on the right-hand side  is bounded in $H^2$ by $CE_d^{1/2},$ since by Lemma \ref{lemma7},
\[{}
\left\|\left[H,\frac{1}{z_\alpha^\prime}\right]\left(z_\alpha^\prime\left(\gamma^\prime(\frac{ z_{t\alpha}}{z_\alpha}-\frac{ z_{t\alpha}^\prime}{z_\alpha^\prime})\right)_\alpha\right)\right\|_{3/2}
\le\left\|\frac{1}{z_\alpha^\prime}\right\|_{4}\left\|\left(z_\alpha^\prime\left(\gamma^\prime(\frac{ z_{t\alpha}}{z_\alpha}-\frac{ z_{t\alpha}^\prime}{z_\alpha^\prime})\right)_\alpha\right)\right\|_{-1}.
\]
We omit the remaining details for our estimate of $B_2$ since they are similar to the above.
We now make the conclusion that
\begin{align}\nonumber
\|B_2\|_{3/2}\le  CE_d^{1/2}+C(|\rho_0-\rho_0^\prime|+|\sigma-\sigma^\prime|).
\end{align}
Finally, $\widetilde{R}_1$ contains as
the highest derivative terms $\gamma_{\alpha}$ and $\theta_{\alpha}.$
So the inequality \eqref{R1Lip} follows from Lemma \ref{lemma2} and all previous estimates.
\end{proof}

Now we prove Theorem \ref{Cauchytheorem}.
\begin{proof}
We take the time derivative of $Z_0$:
\[
\frac{dZ_0}{dt}=\delta L\delta L_t+\int_0^{2\pi}\delta\theta\delta\theta_td\alpha+\int_0^{2\pi}\delta\gamma\delta\gamma_td\alpha
\]
By Lemma \ref{Lippriori}, we could  conclude that
\begin{equation}\nonumber
\|\delta\theta_t\|_0\le CE_d^{1/2} \text{and}~ |\delta L_t|\le CE_d^{1/2}.
\end{equation}
For the estimate of $\int_0^{2\pi}\delta\gamma\delta\gamma_td\alpha$, we must again understand the regularity of $\gamma_{t}.$
First, using the evolution equation \eqref{gammaend2}, we write $\gamma_t$ as
\begin{align}\label{inversegamma}
\gamma_t=\left(I+\left(I+\frac{2\pi}{L}\rho\widetilde{A}\Lambda\right)^{-1}T[\theta]\right)^{-1}\left(I+\frac{2\pi}{L}\rho\widetilde{A}\Lambda\right)^{-1}F.
\end{align}
Since $\theta\in H^s,\gamma\in H^{s-3/2},\tilde{R}_1\in H^{s-2},\tilde{R}_2\in H^{s-3/2}$ and $R\in H^{s-1}$, we have $F\in H^{s-4}$, where $F$ is defined by the formula \eqref{FDef}.
Per our discussion on solvability and Assumption \ref{solvabilityAssumption}, we see that $\gamma_t\in H^0$ when $F\in H^{0}$ for sufficiently large $s$.
Next, using the estimate \eqref{TDef}, we can demonstrate that $T[\theta]\gamma_t\in H^{s-2}$.
Since $\gamma_t=\left(I+\frac{2\pi}{L}\rho\widetilde{A}\Lambda\right)^{-1}\left(-T[\theta]\gamma_t+F\right)$ by  the evolution equation \eqref{gammaend2}, we see that $\gamma_t\in H^{s-4}$.

Now we focus on the  Lipschitz estimate for $T[\theta]$.
Since we can decompose $T[\theta]$ as commutators and terms involving $K[z_d]$, by inequalities \eqref{Lip1} and \eqref{Lip2} we can claim, for any $f\in H^2,$
\begin{equation}\label{LpT}
\|T[\theta]f-T[\theta^\prime]f\|_1\le C E_d^{1/2}\|f\|_2+C|\rho_0-\rho_0^\prime|.
\end{equation}
The proof of \eqref{LpT} is analogous to the proof of Lemma \ref{Tbound} by adding and subtracting several times, so we only give one interesting term as an example:
\begin{multline*}
\left\|K[z_d]\left(\left(\frac{f}{z_\alpha}\right)_{\alpha }\right)-K[z^\prime_d]\left(\left(\frac{f}{z^\prime_\alpha}
\right)_{\alpha }\right)\right\|_1
\le \left\|K[z_d]\left(\left(\frac{f}{z_\alpha}\right)_{\alpha }\right)
-K[z^\prime_d]\left(\left(\frac{f}{z_\alpha}\right)_{\alpha }\right)\right\|_1   \\
+
\left\|K[z^\prime_d]\left(\left(\frac{f}{z_\alpha}\right)_{\alpha }-\left(\frac{f}{z^\prime_\alpha}\right)_{\alpha }\right)\right\|_1\\
\le C\|\theta-\theta^\prime\|_1\left\|\left(\frac{f}{z_\alpha}\right)_{\alpha}\right\|_1
+C\|z_d\|_3\left\|\left(\frac{f}{z_\alpha}\right)_{\alpha }-\left(\frac{f}{z^\prime_\alpha}\right)_{\alpha }\right\|_0
\le C E_d^{1/2}\|f\|_2.
\end{multline*}
Using the evolution equation \eqref{inversegamma}, we subtract to obtain the following equation:
\begin{align*}
\gamma_t-\gamma_t^\prime&=D_1^{-1}D_2^{-1}F-(D_1^\prime)^{-1}(D_2^\prime)^{-1}F^\prime=
(D_1^{-1}-(D_1^\prime)^{-1})(D_1)^{-1}F\\&+(D_1^\prime)^{-1}(D_2^{-1}-(D_2^\prime)^{-1})F+(D_1^\prime)^{-1}(D_2^\prime)^{-1}(F-F^\prime),
\end{align*}
where $D_1=I+\left(I+\frac{2\pi\rho\widetilde{A}\Lambda}{L}\right)^{-1}T[\theta],$ $D_2=I+\frac{2\pi\rho\widetilde{A}\Lambda}{L}$. We are ready to estimate $\int_0^{2\pi}\delta\gamma\delta\gamma_td\alpha.$
\begin{multline}\label{dgammat}
\int_0^{2\pi}\delta\gamma\delta\gamma_td\alpha=\int_0^{2\pi}(D_1^\prime)^{-1}(D_2^\prime)^{-1}(F-F^\prime)\delta\gamma d\alpha\\
+\int_0^{2\pi}\left((D_1^{-1}-(D_1^\prime)^{-1})D_1^{-1}F+(D_1^\prime)^{-1}(D_2^{-1}-(D_2^\prime)^{-1})F
\right)\delta\gamma d\alpha.
\end{multline}

To estimate the first term on the right-hand side of \eqref{dgammat},  we ignore the lower-order derivative terms of $F-F^\prime$ , and focus on the leading term,
\begin{multline*}
-\sigma \bar{A}\int_0^{2\pi}\delta\gamma (D_1^\prime)^{-1}(D_2^\prime)^{-1}\partial_\alpha^4\delta \theta d\alpha\\
\le-\sigma \bar{A}\int_0^{2\pi}D_1^\prime\delta\gamma (D_2^\prime)^{-1}\partial_\alpha^4\delta \theta d\alpha
\le\sigma \bar{A}\int_0^{2\pi}\partial_\alpha (D_1^\prime\delta\gamma)(D_2^\prime)^{-1}\partial_\alpha^3\delta \theta d\alpha.
\end{multline*}
Using the fact that
 $D_1^\prime$ is uniformly bounded from $H^1$ to $H^0$and $(D_2^\prime)^{-1}$ is uniformly bounded from $H^0$ to $H^0$, we obtain
\begin{multline*}
\int_0^{2\pi}(D_1^\prime)^{-1}(D_2^\prime)^{-1}(F-F^\prime)\delta\gamma d\alpha\le CE_d+C(|\rho_0-\rho_0^\prime|+|\sigma-\sigma^\prime|)E_d^{1/2}.
\end{multline*}
To estimate the second and third terms on the right-hand side of \eqref{dgammat}, we use the facts that
\begin{align*}
D_1^{-1}-(D_1^\prime)^{-1}=D_1^{-1}(D_1^\prime-D_1)(D_1^\prime)^{-1},\quad
D_2^{-1}-(D_2^\prime)^{-1}=D_2^{-1}(D_2^\prime-D_2)(D_2^\prime)^{-1}.
\end{align*}
Moreover, we have
\begin{align*}
D_1-D_1^\prime&=\left(I+\frac{2\pi\rho\widetilde{A}\Lambda}{L}\right)^{-1}(T[\theta]-T[\theta^\prime])\\&+2\pi\widetilde{A}\left(I+\frac{2\pi\rho\widetilde{A}\Lambda}{L}\right)^{-1}\left(\frac{\rho}{L}-\frac{\rho^\prime}{L^\prime}\right)\Lambda\left(I+\frac{2\pi\rho\widetilde{A}\Lambda}{L}\right)^{-1}T[\theta^\prime],
\end{align*}
and
\begin{align}\nonumber
D_2-D_2^\prime=2\pi\widetilde{A}\left(\frac{\rho}{L}-\frac{\rho^\prime}{L^\prime}\right)\Lambda.
\end{align}
Using these formulas and the fact that each of $D_2^{-1}$ and $(D_2^\prime)^{-1}$ are uniformly bounded from $H^0$ to $H^0$,
we can conclude that
\begin{equation}\label{dz0-2}
\frac{dZ_0}{dt}\le CE_d+C(|\rho_0-\rho_0^\prime|+|\sigma-\sigma^\prime|)E_d^{1/2}.
\end{equation}

Now, we will take the time derivative of $\sigma Z_1:$
\begin{align*}
\frac{d(\sigma Z_1)}{dt}=\frac{(L^2 \sigma\bar{A})_t}{4\pi^2 }\int_0^{2\pi}(\partial_\alpha^3\delta\theta)^2 d\alpha+\frac{L^2 \sigma\bar{A}}{2\pi^2 }\int_0^{2\pi}(\partial_\alpha^3\delta\theta)
(\partial_\alpha^3\delta\theta_t) d\alpha.
\end{align*}
We apply $\partial_\alpha^3$ to \eqref{evolution of deltatheta}, finding
\begin{equation}\nonumber
\partial_\alpha^3\delta\theta_t=\frac{2\pi^2}{L^2}H(\partial_\alpha^4\delta\gamma)
+\frac{2\pi}{L}V_W^\prime\partial_\alpha^4\delta\theta+\frac{2\pi}{L}\theta_\alpha\partial_\alpha^3\delta V_W+Y_1,
\end{equation}
where $Y_1$ is given by
\begin{equation*}
Y_1=\partial_\alpha^3 B_1+\frac{2\pi}{L}\left(\partial_\alpha^3(V_W^\prime\delta\theta_\alpha)-V_W^\prime\partial_\alpha^4\delta\theta\right)
+\frac{2\pi}{L}\left(\partial_\alpha^3(\theta_\alpha\delta V_W)-\theta_\alpha\partial_\alpha^3\delta V_W\right).
\end{equation*}
Therefore we have
\begin{multline}\label{dz1}
\frac{d(\sigma Z_1)}{dt}
=\frac{(L^2 \sigma\bar{A})_t}{4\pi^2 }\int_0^{2\pi}(\partial_\alpha^3\delta\theta)^2 d\alpha
+\sigma\bar{A}\int_0^{2\pi}(\partial_\alpha^3\delta\theta)
\left(H\partial_\alpha^4\delta\gamma\right) d\alpha
\\+ \frac{L \sigma\bar{A}}{\pi}\int_0^{2\pi}(\partial_\alpha^3\delta\theta)
\left(V_W^\prime \partial_\alpha^4\delta\theta\right) d\alpha+ \frac{L \sigma\bar{A}}{\pi}\int_0^{2\pi}\theta_\alpha(\partial_\alpha^3\delta\theta)
\left(\partial_\alpha^3\delta V_W \right) d\alpha
\\+\frac{L^2 \sigma\bar{A}}{2\pi^2 }\int_0^{2\pi}(\partial_\alpha^3\delta\theta)
(Y_1) d\alpha.
\end{multline}
The third term of the right-hand side of \eqref{dz1} may be bounded in terms of $E_{d}$ upon integrating by parts. To deal with the fourth term and fifth term of the 
right-hand side of \eqref{dz1},
we consider the leading term from $\partial_{\alpha}^{3}V_{W}$,
\begin{equation}\nonumber
\partial_\alpha^{3}\delta V_W=\frac{\pi}{L}\theta_\alpha H(\partial_{\alpha}^{2}\delta\gamma)+Y_{2},
\end{equation}
with remainder $Y_2$  given by
\begin{multline}\nonumber
Y_2=\partial_\alpha^2(\frac{\pi}{L}H(\gamma^\prime \theta_\alpha) -\frac{\pi}{L^\prime}H( \gamma^\prime\theta_\alpha^\prime)+\mathds{P}(\bm{m}\cdot\bm{t}-\mathds{P}(\bm{m}^\prime\cdot\bm{t}^\prime+[H,\theta_\alpha]\delta\gamma)\\
+\frac{\pi}{L}\left(\partial_\alpha^3\theta H(\delta\gamma)+2\partial_\alpha^2\theta  H(\partial_\alpha \delta \gamma )\right).
\end{multline}
 Moreover, we have the estimate
\[
 \frac{L \sigma\bar{A}}{\pi}\int_0^{2\pi}(\partial_\alpha^{3}\delta\theta)
\left( \theta_\alpha Y_2\right) d\alpha
+\frac{L^2 \sigma\bar{A}}{2\pi^2 }\int_0^{2\pi}(\partial_\alpha^3\delta\theta)
(Y_1) d\alpha \le CE_{d}.
\]
We now conclude that
\begin{align}\label{dz1-2}
\frac{dZ_1}{dt}
\le \sigma\bar{A}\int_0^{2\pi}(\partial_\alpha^3\delta\theta)
\left(H\partial_\alpha^4\delta\gamma\right) d\alpha
+ \sigma\bar{A}\int_0^{2\pi}(\partial_\alpha^3\delta\theta)
(H\partial_\alpha^2\delta \gamma )\theta_\alpha^2 d\alpha
+CE_d.
\end{align}
Here, the two terms which need to be treated carefully on the right-hand side come from the second term on the right-hand side of \eqref{dz1},
and the leading-order contribution from the fourth term on the right-hand side of \eqref{dz1}. Adopting a similar approach as we did for  $\frac{dZ_1}{dt},$  we can calculate and estimate $\frac{dZ_3}{dt}:$
\begin{align}
\frac{dZ_3}{dt}\label{dz3-2}
\le \lambda\int_0^{2\pi}(\partial_\alpha^2\delta\theta)
\left(H\partial_\alpha^3\delta\gamma\right) d\alpha- \frac{\sigma \bar{A}}{2}\int_0^{2\pi}(\partial_\alpha^3\delta\theta)
\left(H\partial_\alpha^2\delta\gamma\right) d\alpha+CE_d.
\end{align}
Before we compute $ \frac{d Z_2}{dt},$ we apply $\partial_{\alpha}$ to \eqref{deltagammaend}, finding
\begin{multline}\nonumber
\partial_\alpha\delta\gamma_t=\lambda\partial_\alpha^3\delta\theta-\sigma\bar{A} \partial_\alpha^5\delta\theta -\frac{3\sigma\bar{A}\theta_\alpha^2}{2}\partial_\alpha^3\delta\theta +\left(\frac{2\pi V_W}{L}-\frac{4A\pi^2}{L^2}\gamma\right)\delta\gamma_{\alpha\alpha}+\frac{4\pi\rho\widetilde{A} V_W ^2}{L}\partial_\alpha^3\delta\theta\\
-\frac{2\rho\widetilde{A}\pi}{L}H(\partial_\alpha^2\delta\gamma_{ t})
-\frac{4\rho\widetilde{A}V_W\pi^2}{L^2} H(\partial_\alpha^3\delta\gamma)+\partial_\alpha (T[\theta]\delta\gamma_t)+\rho\partial_\alpha(\widetilde{R}_1-\widetilde{R}_1^\prime)\\+\partial_\alpha B_2 +\rho\partial_\alpha \tilde{B}_2 -\frac{3\sigma\bar{A}(\theta_\alpha^2)_\alpha\partial_\alpha^2\delta\theta }{2} +\partial_\alpha\left(\frac{2\pi V_W}{L}-\frac{4A\pi^2}{L^2}\gamma\right)\delta\gamma_{\alpha}++\frac{4\pi\rho\widetilde{A} (V_W ^2)_\alpha}{L}\partial_\alpha^2\delta\theta.
\end{multline}
We now take the time derivative of $Z_2:$
 \[
 \frac{dZ_2}{dt}=\int_0^{2\pi}(\partial_\alpha \delta\gamma_t)( H\partial_\alpha^{2} \delta\gamma) d\alpha.
 \]
The estimate of $\frac{dZ_2}{dt}$ is analogous to the estimate of $\frac{dE_2}{dt}$ using the skew-adjointedness property of the Hilbert transform,  commutators, and integration by parts.  We therefore only give the leading terms of $dZ_2/dt:$
\begin{multline}\label{dz2-2}
\frac{dZ_2}{dt}\le-\frac{2\rho\widetilde{A}\pi}{L}\int_0^{2\pi}H(\partial_\alpha^2\delta\gamma_{ t})( H\partial_\alpha^{2} \delta\gamma) d\alpha-\sigma\bar{A}\int_0^{2\pi}\partial_\alpha^5\delta\theta( H\partial_\alpha^{2} \delta\gamma) d\alpha\\
+\int_0^{2\pi}\left(\lambda-\frac{3\sigma\bar{A}\theta_\alpha^2}{2}+\frac{4\pi\rho\tilde{A}V_W^2}{L}\right)\partial_\alpha^3\delta\theta( H\partial_\alpha^{2} \delta\gamma) d\alpha+CE_d+C|\rho_0-\rho_0^\prime|E_d^{1/2}\\
\le-\frac{d(\rho Z_4)}{dt}
-\sigma\bar{A}\int_0^{2\pi}\partial_\alpha^3\delta\theta( H\partial_\alpha^{4} \delta\gamma) d\alpha
+C E_d+C(|\rho_0-\rho_0^\prime|+|\sigma-\sigma^\prime|)E_d^{1/2}\\
+\int_0^{2\pi}\left(\lambda-\frac{3\sigma\bar{A}\theta_\alpha^2}{2}+\frac{4\pi\rho\tilde{A}V_W^2}{L}\right)\partial_\alpha^3\delta\theta( H\partial_\alpha^{2} \delta\gamma) d\alpha.
\end{multline}
Adding \eqref{dz0-2}, \eqref{dz1-2}, \eqref{dz3-2}, and \eqref{dz2-2}, we conclude that
\begin{multline}\label{zdend}
\frac{dZ_d}{dt}\le  \int_0^{2\pi}\left(-\frac{\sigma\bar{A}(\theta_\alpha^2+1)}{2}+\frac{4\pi\rho\tilde{A}V_W^2}{L}\right)\partial_\alpha^3\delta\theta H\partial_\alpha^{2} \delta\gamma d\alpha
\\
+C E_d +C(|\rho_0-\rho_0^\prime|+|\sigma-\sigma^\prime|)E_d^{1/2}.
\end{multline}
To cancel the first term on the right-hand side of the above inequality, we introduce additional energy terms:
\begin{align}\nonumber
Z_5=\frac12\int_0^{2\pi}\sqrt{(\theta_\alpha^2+1)/2}(\delta \gamma)\Lambda\left(\sqrt{(\theta_\alpha^2+1)/2}(\delta\gamma)\right)d\alpha,
\end{align}
\begin{align}\nonumber
Z_6=\int_0^{2\pi}\frac{(\theta_\alpha^2+1)\widetilde{A}\pi}{2L}(H\partial_\alpha\delta\gamma)^2d\alpha,
\end{align}
\begin{align}\nonumber
Z_7=\frac{ L\tilde{A}}{\pi}\int_0^{2\pi}V_W^2(\partial_\alpha^{2}\delta\theta)^2d\alpha.
\end{align}
We omit most remaining details but give the leading terms for $dZ_5/dt,$
\begin{multline*}\label{dZ42}
\frac{dZ_5}{dt}\le \int_0^{2\pi}\frac{(\theta_\alpha^2+1)}{2}(\delta\gamma_t)(H\partial_\alpha\delta\gamma)d\alpha+C E_d +C(|\rho_0-\rho_0^\prime|+|\sigma-\sigma^\prime|)E_d^{1/2}\\
\le
-\int_0^{2\pi}\frac{(\theta_\alpha^2+1)\rho\widetilde{A}\pi}{L}H(\partial_\alpha\delta\gamma_t)(H\partial_\alpha\delta\gamma)d\alpha\\
+\int_0^{2\pi}\frac{\sigma\bar{A}(\theta_\alpha^2+1)}{2}(\partial_\alpha^{3}\delta\theta)(H\partial_\alpha^{2}\delta\gamma)d\alpha+C E_d+C(|\rho_0-\rho_0^\prime|+|\sigma-\sigma^\prime|)E_d^{1/2}\\
\le -\frac{d(\rho Z_6)}{dt}+\int_0^{2\pi}\frac{\sigma\bar{A}(\theta_\alpha^2+1)}{2}(\partial_\alpha^{3}\delta\theta)(H\partial_\alpha^{2}\delta\gamma)d\alpha\\
+C E_d +C(|\rho_0-\rho_0^\prime|+|\sigma-\sigma^\prime|)E_d^{1/2},
\end{multline*}
and $d(\rho Z_7)/dt,$
\begin{align}
\frac{d(\rho E_7)}{dt}&=\int_0^{2\pi} \left(\frac{2L\rho\tilde{A}V_W^2}{\pi}\right)_t\left(\partial_\alpha^{2} \delta\theta\right)\partial_\alpha^{2}\delta\theta d\alpha+\int_0^{2\pi} \frac{2L\rho\tilde{A}V_W^2}{\pi}\left(\partial_\alpha^{2}\delta \theta\right)\partial_\alpha^{2}\delta\theta_t d\alpha \nonumber\\
&\le \int_0^{2\pi} \frac{4\pi\rho\tilde{A}V_W^2}{L}\left(\partial_\alpha^{2}\delta \theta\right)(H\partial_\alpha^{3}\delta\gamma) d\alpha+CE_d\nonumber\\
&\le-\int_0^{2\pi} \frac{4\pi\rho\tilde{A}V_W^2}{L}\left(\partial_\alpha^{3} \delta\theta\right)(H\partial_\alpha^{2}\delta\gamma) d\alpha+CE_d,
\nonumber
\end{align}
which provides a cancellation with the first term on the right-hand side of \eqref{zdend}.

We define $E_d^{total}$ as
\begin{align*}\nonumber
E_d^{total}&=Z_0+\sigma Z_1+Z_2+Z_3+\rho Z_4+Z_5+\rho Z_6+\rho Z_7\\
&=Z_0+\sigma Z_1+Z_2+Z_3++2\pi\rho_0 Z_4/L+Z_5+2\pi\rho_0Z_6/L+\rho_0 Z_7/L.
\end{align*}
Then we have the following estimate:
\begin{align}\nonumber
\frac{dE_d^{total}}{dt}\le C_1E_d^{total}+C_2(|\rho_0-\rho_0^\prime|+|\sigma-\sigma^\prime|)(E_d^{total})^{1/2},
\end{align}
where $C_1$ and $C_2$ are independent of $\sigma$ and $\rho_{0}.$
Solving the differential inequality, we see that
\begin{equation}\nonumber
E_d^{total}\le E_d^{total}(0)e^{C_1 t}+C_2(|\rho_0-\rho_0^\prime|+|\sigma-\sigma^\prime|)(e^{C_1 t}-1)/C_1.
\end{equation}
We know that the solutions start with the same initial conditions, so that $ E_d^{total}(0)=0.$  This completes the proof.
\end{proof}

From the paper \cite{ambroseThesis}, we know that the two-dimensional vortex sheet with surface tension is well-posed. That is, for
the system \eqref{evolution of theta},
\eqref{gammaend00}, when $(\sigma,\rho_0)=(0,0)$, there exists a bounded solution $(\theta,\gamma)\in C^0([0,T^\prime];\overline{\mathcal{O}}).$
Now we will prove that this two-dimensional vortex sheet with surface tension
is the limit of the two-dimensional hydroelastic wave (with surface tension), as the density of the elastic sheet and bending modulus vanish.
This is done while assuming that Assumption \ref{solvabilityAssumption} holds.   This is the content of Theorem \ref{limitTheorem}; before stating the theorem,
we introduce some brief notation.
First, we write the evolution of $\gamma$ as
\begin{equation}\label{gammaend3}
\gamma_t=\left(I+\frac{2\pi}{L}\rho\widetilde{A}\Lambda\right)^{-1}(-T[\theta](\gamma_t)+F)=B_1^{\sigma,\rho_0}
\end{equation}
and the evolution of $\theta$
\begin{equation}\nonumber
\theta_t=B_2^{\sigma,\rho_0}.
\end{equation}
\begin{theorem}\label{limitTheorem}
Let $(\theta_0,\gamma_0)\in \mathcal{O}$ be given, with $\llangle\sin(\theta_0)\rrangle=0$. Assume that Assumption \ref{solvabilityAssumption} holds.
Let $T>0$ be the value guaranteed to exist (independent of small $\sigma$ and $\rho_{0}$) from Theorem \ref{uniform_time}.  
Let   $(\theta^{\sigma,\rho_0},\gamma^{\sigma,\rho_0}) \in C([0,T],\mathcal{O})$ be the solution of Cauchy problem \eqref{evolution of theta}, \eqref{gammaend} 
with initial conditions  $(\theta_0,\gamma_0).$   Let $(\theta,\gamma)\in C([0,T];\bar{\mathcal{O}})$ be the corresponding solution of the initial value problem
for $(\sigma,\rho_{0})=(0,0).$
For any $s^\prime$ such that $0\le s^\prime<s$, we have
\[
\lim_{(\sigma,\rho_0)\rightarrow(0,0)}\sup_{t\in[0,T]}\|(\theta^{\sigma,\rho_0},\gamma^{\sigma,\rho_0})-(\theta,\gamma)\|_{H^{s^\prime-1}\times H^{s^\prime-3/2}}=0.
\]
\end{theorem}
\begin{proof}
From Theorem \ref{Cauchytheorem}, we see that $(\theta^{\sigma,\rho_0},\gamma^{\sigma,\rho_0})$ is a Cauchy sequence in $H^2\times H^{3/2}$. 
Since the solutions $(\theta^{\sigma,\rho_0},\gamma^{\sigma,\rho_0})$ are uniformly bounded in $H^{s-1}\times H^{s-3/2}$  regardless of $(\sigma,\rho_0)$, 
by the Sobolev interpolation inequality, it implies that the sequence   $(\theta^{\sigma,\rho_0},\gamma^{\sigma,\rho_0})$ is a Cauchy sequence in 
$H^{s^\prime-1}\times H^{s^\prime-3/2}$. Therefore, there exists a limit as $(\sigma,\rho_0)\rightarrow(0,0)$.

Now we will prove that the limit is $(\theta,\gamma),$ the solution of the Cauchy problem for the vortex sheet with surface tension. 
Since $s$ is sufficiently large we have 
from  \eqref{gammaend3} that ${B}_1^{\sigma,\rho_0}$ converges uniformly to $B_1^{(\sigma,\rho_0)=(0,0)}$. We integrate \eqref{gammaend3} in time,
\[
\gamma^{\sigma,\rho_0}(\cdot,t)=\gamma_0+\int_0^t{B}_1^{\sigma,\rho_0}(\cdot,s)ds,
\]
and pass to the limit as $(\sigma,\rho_0)\rightarrow (0^+,0^+).$ Using the uniform convergence ${B}_1^{\sigma,\rho_0}$ to exchange the 
limit and the integral, we have
\[
\gamma(\cdot,t)=\gamma_0+\int_0^t\mathcal{B}_1^{(\sigma,\rho_0)=(0,0)}(\cdot,s)ds.
\]
Similarly, again since $s$ is taken sufficiently large, 
${B}_2^{\sigma,\rho_0}$ converges uniformly to $B_2^{(\sigma,\rho_0)=(0,0)}.$ 
This is again enough to conclude $\theta^{\sigma,\rho_0}$ converges uniformly to $\theta.$
This implies that $(\theta,\gamma)$ solves the Cauchy problem of the vortex sheet with surface tension.
\end{proof}

\section*{Acknowledgments}
The first author was supported in part by  the National Natural Science Foundation of China (Grant No.12001187) and the Hunan Provincial Educational Foundation of China (Grant No.23B0569). DMA is grateful to the National Science Foundation for support through grant DMS-2307638.

\bibliography{hydroelastic-with-mass}{}
\bibliographystyle{plain}

\end{document}